%% file: paper.tex
\documentclass[final]{siamart251216}

\usepackage{amssymb,amsmath,amsfonts,amscd,verbatim}
\usepackage{color,graphicx}
\usepackage{algorithm,algorithmic}
\usepackage{upgreek,comment,url}
\usepackage{float}
\usepackage{rotating,fancyhdr}
\usepackage{booktabs}

\usepackage{mystylefile}

\title{Revisiting column subset selection\\ through the lens of submodularity\thanks{The work was funded, in part,
by the National Science Foundation through awards DMS-2411198 and CCF-2209510, and
by the Department of Energy through awards DE-SC0023188 and DE-SC0026310.}}

\author{Ilse C.F.\ Ipsen\thanks{Department of Mathematics, North Carolina State University, Raleigh, NC 27695, USA. \email{ipsen@ncsu.edu}, \email{asaibab@ncsu.edu}} \and Arvind K.\ Saibaba\footnotemark[2] }

\begin{document}
\maketitle

\begin{abstract}
The problem is to select $k$ columns with maximal volume from a real $m\times n$ matrix~$\mx$. We show
that the logarithm of the volume is
a set submodular function on columns of $\mx$, and
for full column-rank matrices $\mx$ with sufficiently large singular values, it is a non-negative non-decreasing function. As a consequence, 
traditional Businger-Golub QR with column pivoting is a greedy algorithm, with a relative
error of at most 37 percent. In contrast, Gu-Eisenstat strong rank-revealing
QR is a 1-interchange algorithm, with a relative
error of at most 50 percent.
The higher accuracy, under this metric, of the simple QR with column pivoting  
confirms its well known effectiveness in practice.
For general, possibly rank-deficient matrices, we derive probabilistic bounds for the absolute error based on a smoothed analysis.
The above analyses are extended to finding $k\times k$ submatrices of maximal volume in symmetric positive-definite
matrices.
\end{abstract}



\input{sec1}
\input{sec2}

\input{sec3}

\input{sec4}

\input{sec5}

\input{sec6}

\input{sec7}
\subsection*{Acknowledgments}
The authors would like Daniel Kressner for raising an issue with an earlier version of this manuscript. AKS would also like to thank Hugo D\'iaz for useful conversations.
\input{sec8}
\bibliography{ssbib}
\bibliographystyle{siam}

\end{document}

%% file: sec1.tex

\section{Introduction}
Many problems in 
combinatorial optimization \cite{Nemhauser1978}, machine learning \cite{DasK2018,Mirza2013}, 
management science \cite{Corne1977}, determinantal point processes \cite{ChenZhang2018,Gillen12},
matroids \cite{Cali11}, sensor placement \cite{Hash2020,Sham2010}, and D-optimal design \cite{Rob89}
can be expressed as the maximization 
of a submodular function under cardinality constraints. 
Such maximization problems can be NP-complete,
such as the account location problem \cite[page 791]{Corne1977}, and
heuristics have been developed, including greedy 
and interchange algorithms that come with
relative error bounds \cite{Nemhauser1978}.

To take advantage of these error bounds in the
context of selecting $k$ 'good' columns from a
matrix $\mx\in\rmn$, we
express column subset selection in terms of volume, and show that the log-volume (logarithm of the volume) is 
a set submodular function on column submatrices
of $\mx$.

For matrices with full column-rank and sufficiently large singular values, 
we show that traditional Businger-Golub QR with column pivoting~\cite{BusGo65,GoVL13} is a greedy algorithm 
for log-volume maximization, with a relative
error of at most 37 percent.
In contrast, Gu-Eisenstat strong rank-revealing
QR~\cite{GuEis96} is a 1-interchange algorithm
for log-volume maximization, with a relative
error of at most 50 percent---however, only in theory,
when run without a relaxation factor. 
The higher accuracy, under this metric, of the simple QR with column pivoting  
confirms its well known empirical effectiveness,
and competitiveness with more sophisticated 
pivoting strategies.

For general, possibly rank-deficient matrices, however
the log-volume can be negative and decreasing, hence the error
bounds in \cite{Nemhauser1978} do not apply.
As an alternative, we derive absolute error bounds via a smoothed analysis, in which the input matrix undergoes a small perturbation so that the perturbed matrix has full column-rank with high probability. Although the algorithms are applied to the perturbed matrix, our analysis presents error bounds for
the original matrix.

\subsection{Contributions} 
We express existing algorithms for column subset selection and submatrix selection
as greedy and 1-interchange algorithms for submodular
functions, and derive error bounds on their accuracy.
\begin{enumerate}
\item We introduce variants of 1-interchange algorithms for 
set submodular functions with additive and multiplicative relaxation factors,
and derive error bounds for their accuracy (Theorem~\ref{t_53}).
\item We show that the log-volume of a general matrix is a 
set submodular function (Theorem~\ref{t_logvol}).
\item We show that Businger-Golub QR with column 
pivoting \cite{BusGo65,GoVL13}
is a greedy algorithm for log-volume maximization
for matrices that have full column-rank (Theorem~\ref{t_2} and Corollary~\ref{c_2}), 
with a relative error of at most 37 percent. 
\item We show that Gu-Eisenstat strong 
rank-revealing QR 
\cite{GuEis96} is a 1-interchange algorithm for
log-volume maximization for matrices that have full column rank (Theorem~\ref{t_3},
Corollary~\ref{c_3} and Theorem~\ref{c_54}), with a relative error of at most 50 percent. 

According to this metric, the traditional 
Businger-Golub QR with column pivoting can be more
accurate than the strong rank-revealing Gu-Eisenstat
QR factorization. This explains the accuracy
of QR with column pivoting frequently observed in practice.

\item For general matrices, which may be rank-deficient, we derive absolute error bounds for Businger-Golub QR with column pivoting (Theorem~\ref{thm:smoothed}) and Gu-Eisenstat 
strong rank-revealing QR (Theorem~\ref{thm:smoothed2}) via a smoothed analysis.  Although the algorithms are applied to the perturbed matrix, our analysis presents error bounds for
the original matrix. 

\item For symmetric positive definite matrices, we show that pivoted Cholesky factorizations 
\cite{GoVL13,gu2004strong,Higham2002} and strong rank-revealing Cholesky factorization~\cite{gu2004strong}
are, respectively, greedy and 1-interchange algorithms for log-volume maximization of principal submatrices,
and derive relative (Theorem~\ref{thm:pivcholesky} and Corollary~\ref{c_100}) and absolute 
(Theorem~\ref{thm:srrcholesky}) error bounds. 
A smoothed analysis extends the error bounds to positive semidefinite matrices 
(Theorems~\ref{thm:smoothed_psd} and~\ref{thm:smoothed_psd2}).
\end{enumerate}

\subsection{Overview}
After reviewing submodular functions and their algorithms
(Section~\ref{s_submod}), we discuss submodular
functions for column subset selection (Section~\ref{s_3}). 
Then we express and analyze the Businger-Golub QR factorization as a greedy algorithm (Section~\ref{s_4}); the Gu-Eisenstat QR factorization as a 1-interchange algorithm (Section~\ref{sec:geqr}); and pivoted Cholesky factorizations 
as greedy and 1-interchange algorithms for submatrix
selection (Section~\ref{s_6}); followed by 
possible directions of future work
(Section~\ref{s_conc}). An appendix presents the proofs for the smoothed analysis.

\subsection{Notation}
Let $\mx\in\rmn$ be  matrix with 
singular values 
$\sigma_1(\mx)\geq \cdots\geq\sigma_{\min\{m,n\}}(\mx)\geq 0$.
For some $0<k\leq \rank(\mx)$,
let $\mc\in\real^{m\times k}$ 
be a submatrix of $k$ columns of $\mx$, with singular values 
\begin{equation}\label{e_so}
\sigma_1(\mc)\geq \cdots\geq\sigma_k(\mc)\geq 0.
\end{equation}
Singular value interlacing \cite[Corollary 7.3.6 and 7.3.P44]{HoJoI}
implies
\begin{equation}\label{e_inter}
\sigma_{\min\{m,n\}-k+j}(\ma)\leq \sigma_j(\mc)\leq
    \sigma_j(\ma),\qquad 1\leq j\leq k.
\end{equation}
For matrices $\mc, \md \in \real^{m\times n}$, Weyl's inequality~\cite[(7.3.15)]{HoJoI} for singular values implies
\begin{equation}\label{eqn:weyl}
\max_{1 \le j \le \min\{m,n\}}|\sigma_j(\mc) - \sigma_j(\md)| \le \|\mc-\md\|_2. 
\end{equation}

The two-norm condition number with respect to left inversion of 
$\mx\in\rmn$ with $\rank(\mx)=n$ is
$\kappa_2(\mx) \equiv  \sigma_1(\mx)/\sigma_n(\mx)$.

The eigenvalues of a symmetric positive semi-definite matrix
$\ma\in\rnn$ are $\lambda_1(\ma)\geq \cdots\geq \lambda_n(\ma)\geq 0$.
For symmetric positive semi-definite matrices $\ma, \mb\in\rmm$, the Loewner inequality $\ma\preceq \mb$ says that
$\mb-\ma$ is positive semi-definite.

For symmetric matrices $\ma,\me \in \real^{n\times n}$, Weyl's monotonicity theorem for eigenvalues~\cite[Corollary 4.3.15]{HoJoI} 
implies
\begin{equation}\label{eqn:weyl2}
    \lambda_i(\ma) + \lambda_n(\me) \le \lambda_i(\ma + \me) \le \lambda_i(\ma) + \lambda_1(\me), \qquad 1 \le i \le n.
\end{equation}

We denote by $\mi_n\equiv\begin{bmatrix} \ve_1& \cdots &\ve_n\end{bmatrix}\in\rnn$ the identity matrix, and by $\ve_j\in\rn$ its columns for $1 \le j \le n$. The norm $\|\cdot\|$ is the Euclidean two-norm.

\begin{definition}\label{d_set}
Let $\mx\in\rmn$ and $\mathcal{S}=\{s_1,\ldots,s_k\}
\subset\{1,\ldots, n\}$ for some $1\leq k\leq n$. Then
$\mx_{\mathcal{S}}=\begin{bmatrix}\vx_{s_1} &\ldots & \vx_{s_k}\end{bmatrix}\in\real^{m\times k}$ is a 
submatrix of $\mx$ with columns indexed by  $\mathcal{S}$.

If $\ma\in\rnn$, then $\ma[\mathcal{S}]\in\real^{k\times k}$ is the principal submatrix of $\ma$ indexed by
rows and columns in $\mathcal{S}$.
\end{definition}

%% file: sec2.tex
\section{Submodular functions}\label{s_submod}
After presenting the problem of maximization under cardinality constraints, and the definition of set submodular functions (Definition~\ref{d_1}), we review greedy algorithms and interchange algorithms
(Section~\ref{s_subalg}). We introduce and analyze a relaxed
version of the interchange algorithms (Section~\ref{s_relaxed}),
and review existing work (Section~\ref{s_2rel}).

\paragraph{Maximization under cardinality constraints}
Given a finite set $\mathcal{X}$, a real valued function $f$ on the subsets of $\mathcal{X}$,
and an integer $0<k<|\mathcal{X}|$, we want to maximize $f$ over all subsets $\mathcal{S}\subset\mathcal{X}$ of cardinality $k$,
that is,
\begin{equation*}
\max_{\mathcal{S}\subset\mathcal{X}}\{f(\mathcal{S}):\ |\mathcal{S}|=k\}.
\end{equation*}

The notion of \emph{set submodularity}, defined below, is `in some sense a combinatorial analogue
of concavity' \cite[Section 1]{Nemhauser1978}.

\begin{definition}[(1.4), (1.5) and Proposition 2.1 in \cite{Nemhauser1978}]\label{d_1}
Given a finite set~$\mathcal{X}$, a real-valued function $f$ on the  subsets of $\mathcal{X}$
is a \emph{set submodular function} if
\begin{equation*}
f(\mathcal{T}\cup \{x\}) - f(\mathcal{T})\leq f(\mathcal{S}\cup\{x\})-f(\mathcal{S})\quad
\text{for all} \quad \mathcal{S}\subset\mathcal{T}\subset\mathcal{X},\quad
x\in\mathcal{X}-\mathcal{T}.
\end{equation*}
Furthermore, a set submodular function $f$ is \emph{non-decreasing} if
\begin{equation*}
f(\mathcal{T}\cup \{x\}) -f(\mathcal{T})\geq 0\quad \text{for all}\quad \mathcal{T}\subset\mathcal{X}, \quad x\in\mathcal{X}-\mathcal{T}. 
\end{equation*}
\end{definition}

That is, adding a new element $x$ to a smaller set $\mathcal{S}$ is more beneficial than adding~$x$ to a larger set $\mathcal{T}$; and adding a new element $x$ can possibly increase the function, but cannot decrease it. 
For instance, linear functions are submodular.

\subsection{Algorithms for set submodularity}\label{s_subalg}
We review two maximization algorithms: 
a greedy algorithm (Algorithm~\ref{alg_proto}) and its accuracy (Theorem~\ref{t_ng}); 
and a 1-interchange algorithm (Algorithm~\ref{alg_proto2}) and its accuracy (Theorem~\ref{t_ng2}).

The greedy algorithm in \cite[Section 4]{Nemhauser1978} is more general than Algorithm~\ref{alg_proto} 
because it applies to general submodular functions $f$ that are not necessarily non-decreasing. However, here we need a non-negative
non-decreasing function $f$ so that we can set $f(\emptyset)=0$.

\begin{algorithm}[!ht]
\caption{Greedy algorithm for non-negative non-decreasing set submodular functions~$f$ 
\cite[Algorithm 1]{DasK2018}, \cite[Section 4]{Nemhauser1978}}\label{alg_proto}
\begin{algorithmic}[1]
\REQUIRE Finite set $\mathcal{X}$, integer $0<k<|\mathcal{X}|$,\\
$\qquad$ Non-negative non-decreasing set submodular function $f$ with $f(\emptyset)=0$

\STATE Initialize $\mathcal{S}_0\equiv\emptyset$
\FOR{$i=0:k-1$}
\STATE Find $x_{i+1}\in\mathcal{X}-\mathcal{S}_{i}$ that maximizes $f(\mathcal{S}_i\cup\{x_{i+1}\}) $
\STATE Set $\mathcal{S}_{i+1}\equiv \mathcal{S}_i\cup\{x_{i+1}\}$
\ENDFOR
\RETURN $\mathcal{S}_k$
\end{algorithmic}
\end{algorithm}

\begin{theorem}[Theorem 1 in \cite{Corne1977}, Theorem 5 in \cite{DasK2018}, Proposition 4.3 in \cite{Nemhauser1978}]\label{t_ng}
Let $f$ be a non-negative non-decreasing set submodular function on a finite set
$\mathcal{X}$ with $f(\emptyset)=0$;  $0<k<|\mathcal{X}|$ an integer; and~$\mathcal{S}_k^*$ a set that maximizes
$f(\mathcal{S})$ over all subsets $\mathcal{S}\subset\mathcal{X}$ with $|\mathcal{S}|=k$. 

Then the set $\mathcal{S}_k$ returned by Algorithm~\ref{alg_proto} satisfies 
\begin{equation*}
0\leq f(\mathcal{S}_k^*)-f(\mathcal{S}_k)
\leq (1-\tfrac{1}{k})^k\,f(\mathcal{S}_k^*),
\end{equation*}
where $(1-\tfrac{1}{k})^k\leq \tfrac{1}{e} <.37$.
\end{theorem}

The upper bound in Theorem~\ref{t_ng} is tight for the uncapacitated location problem \cite[Theorem 3]{Corne1977}, 
\cite[Theorem 4.1(c)]{Nemhauser1978}. For certain classes of problems,~\cite[Theorem 4.2]{nemhauser1978best} shows that the upper bound cannot be improved by an algorithm with a polynomial number of function evaluations.


\begin{algorithm}[!ht]
\caption{$1$-interchange algorithm \cite[Section 5]{Nemhauser1978}}\label{alg_proto2}
\begin{algorithmic}[1]
\REQUIRE Finite set $\mathcal{X}$, subset $\mathcal{S}^{(0)}\subset\mathcal{X}$ with $|\mathcal{S}^{(0)}|=k$\\
$\qquad$ Non-negative non-decreasing set submodular function $f$ with $f(\emptyset)=0$
\STATE $i=0$\\
\WHILE{exist $x_{i+1}\in\mathcal{S}^{(i)}$ and $y_{i+1}\in\mathcal{X}-\mathcal{S}^{(i)}$\\ with
$f((\mathcal{S}^{(i)}-\{x_{i+1}\})\cup\{y_{i+1}\})>f(\mathcal{S}^{(i)})$}
\STATE $\mathcal{S}^{(i+1)}=(\mathcal{S}^{(i)}-\{x_{i+1}\})\cup\{y_{i+1}\}$
\STATE $i= i+1$
\ENDWHILE
\STATE $\mathcal{S}_k\equiv \mathcal{S}^{(i-1)}$
\RETURN $\mathcal{S}_k$
\end{algorithmic}
\end{algorithm}

\begin{theorem}[Theorem 5 in \cite{Corne1977}, Theorem 5.1 in \cite{Nemhauser1978}]\label{t_ng2}
Let $f$ be a non-negative non-decreasing set submodular function on a finite set
$\mathcal{X}$ with $f(\emptyset)=0$;  $0<k<|\mathcal{X}|$ an integer; 
and~$\mathcal{S}_k^*$ a set that maximizes
$f(\mathcal{S})$ over all subsets $\mathcal{S}\subset\mathcal{X}$ with $|\mathcal{S}|=k$.

Then the set $\mathcal{S}_k$ returned by Algorithm~\ref{alg_proto2} satisfies 
\begin{equation*}
0\leq f(\mathcal{S}_k^*)-f(\mathcal{S}_k)
\leq \frac{k-1}{2k-1}\,f(\mathcal{S}_k^*),
\end{equation*}
where $\frac{k-1}{2k-1}<\frac{1}{2}$.
\end{theorem}

The upper bound is tight for the uncapacitated location problem \cite[Theorem 5.1(b)]{Nemhauser1978}.

\subsection{Algorithm~\ref{alg_proto2} with relaxation
factors}\label{s_relaxed}
We adapt the analysis of~\cite[Section 5]{Nemhauser1978} to the case where interchanges take place only if the objective function improves by an additive relaxation
factor $\theta_a \ge 0$, or by a multiplicative relaxation factor $\theta_m \ge 1$. 
For an additive factor, the  condition in the while loop 
in line~2 of Algorithm~\ref{alg_proto2} becomes 
\begin{equation}\label{e_add}
f( (\mathcal{S}^{(i)} - \{x_i\}) \cup \{y_j\} )  > f(\mathcal{S}^{(i)}) + \theta_a. 
\end{equation}
The output $\mathcal{S}_k$ of Algorithm~\ref{alg_proto2} is locally optimal because for every $x \in \mathcal{S}_k$ and $y \in \{1,\ldots,n\} - \mathcal{S}_k$, 
\[ f( (\mathcal{S}_k - \{ x \}) \cup \{y\})  
\le f(\mathcal{S}_k) + \theta_a. \] 
For a multiplicative factor, the condition in the while loop in line~2 of Algorithm~\ref{alg_proto2} becomes
\begin{equation}\label{e_mult}
f( (\mathcal{S}^{(i)} - \{x_i\}) \cup \{y_j\} )  > \theta_m f(\mathcal{S}^{(i)}).
\end{equation}
Again, the output $\mathcal{S}_k$ from Algorithm~\ref{alg_proto2} is locally optimal.

\begin{theorem}\label{t_53}
Let $f$ be a non-negative non-decreasing set submodular function on a finite set
$\mathcal{X}$ with $f(\emptyset)=0$;  $0<k<|\mathcal{X}|$ an integer; and~$\mathcal{S}_k^*$ a set that maximizes
$f(\mathcal{S})$ over all subsets $\mathcal{S}\subset\mathcal{X}$ with $|\mathcal{S}|=k$.
If $\mathcal{S}_k$ is the output of the 1-interchange Algorithm~\ref{alg_proto2} with additive relaxation
factor $\theta_a \ge 0$, then 
\begin{align*}
0\leq f(\mathcal{S}_k^*)-f(\mathcal{S}_k) \leq\frac{k-1}{2k-1}f(\mathcal{S}_k^*)+\frac{k^2}{2k-1}\theta_a,
\end{align*}
where $\frac{k-1}{2k-1}<\frac{1}{2}$ and
$\frac{k^2}{2k-1}\leq k$.

 If $\mathcal{S}_k$ is the output of the 1-interchange Algorithm~\ref{alg_proto2} with multiplicative 
 relaxation factor $\theta_m\geq 1$, then
\begin{align*}
0\leq f(\mathcal{S}_k^*)-f(\mathcal{S}_k) \leq \frac{k-1 +k^2(\theta_m-1)}{2k-1 +k^2(\theta_m-1)} f(\mathcal{S}_k^*). 
\end{align*}
\end{theorem}

In the special case $\theta_a = 0$ or $\theta_m = 1$, Theorem~\ref{t_53} reduces to Theorem~\ref{t_ng2}.

\begin{proof}
We proceed as in the proof of~\cite[Theorem 5.1]{Nemhauser1978}.
    Denote the elements of the optimal solution by $\mathcal{S}_k^* = \{s_1^*,\dots,s_k^*\}$. 

    Apply Algorithm~\ref{alg_proto} to the output $\mathcal{S}_k$ 
    of Algorithm~\ref{alg_proto2}
    to order the elements $\mathcal{S}_k = \{s_1,\dots,s_k\}$, and define the
    subsets $\mathcal{S}_{j} \equiv \{s_1,\dots,s_j\}$, $1\le j \le k$ and $\mathcal{S}_0 \equiv \emptyset$. 
      
    Define the increments $\zeta_j \equiv f(\mathcal{S}_s) - f(\mathcal{S}_{s-1})$, $1 \le j \le k$. Then
    \begin{align*} \zeta_{j} &= f(\mathcal{S}_j) - f(\mathcal{S}_{j-1}) =
     f(\mathcal{S}_{j-1} \cup \{s_j\}) - f(\mathcal{S}_{j-1}) 
    \le  \>  f(\mathcal{S}_{j-2} \cup \{s_j\}) - f(\mathcal{S}_{j-2}) \\ 
    &\le  \>  f(\mathcal{S}_{j-2} \cup \{s_{j-1}\}) - f(\mathcal{S}_{j-2}) =  f(\mathcal{S}_{j-1}) - f(\mathcal{S}_{j-2}) = \zeta_{j-1}, \qquad 2\leq j\leq k,
    \end{align*} 
    where the first inequality follows from submodularity, and the second from greedy ordering.
     Thus, the increments are non-increasing
    $\zeta_1 \ge \dots \ge \zeta_{k}$. With $f(\emptyset)=0$ this implies
    \begin{equation}\label{eqn:decomp} f(\mathcal{S}_k) = f(\emptyset) + \sum_{j=1}^{k}\zeta_j \ge k \zeta_{k}. \end{equation}
From \cite[Proposition 2.1(iv)]{Nemhauser1978} or~\cite[Proposition 2.4]{Nemhauser1978} follows 
    \begin{equation}\label{eqn:decomp2} f(\mathcal{S}_k^*) \le f(\mathcal{S}_{k-1}) + \sum_{s_j^* \in \mathcal{S}_k^* - \mathcal{S}_{k-1}} (f(\mathcal{S}_{k-1} \cup \{s_j^*\} ) - f(\mathcal{S}_{k-1})).\end{equation}
    We consider additive and multiplicative relaxation
    factors separately.
    
    \paragraph{Additive relaxation} The interchange property (\ref{e_add}) ensures
     \begin{equation*}
    f(\mathcal{S}_{k-1} \cup \{s_j^*\} )  \le  f(\mathcal{S}_k) + \theta_a , \qquad s_j^* \in \mathcal{S}_k^* - \mathcal{S}_{k-1}.
    \end{equation*}
    Subtract $f(\mathcal{S}_{k-1})$ on both sides,
    \[ f(\mathcal{S}_{k-1} \cup \{s_j^*\} ) - f(\mathcal{S}_{k-1}) \le  f(\mathcal{S}_k)  - f(\mathcal{S}_{k-1}) +\theta_a = \zeta_{k} +\theta_a , \qquad s_j^* \in \mathcal{S}_k^* - \mathcal{S}_{k-1}.  \]
Insert this inequality into~\eqref{eqn:decomp2}
    \[ \begin{aligned} f(\mathcal{S}_k^*) \le & \> f(\mathcal{S}_{k-1}) + | \mathcal{S}_k^*-\mathcal{S}_{k-1} |(\zeta_{k} + \theta_a) \\
    \le & \> f(\mathcal{S}_{k-1}) + k\zeta_k +  k\theta_a = f(\mathcal{S}_k) + (k-1)\zeta_{k} + k\theta_a,
    \end{aligned}\]
    where the last inequality follows from~\eqref{eqn:decomp} and $1 \le |\mathcal{S}_k^*-\mathcal{S}_{k-1} |\leq k$. From~\eqref{eqn:decomp} follows $\zeta_k \le kf(\mathcal{S}_k)$, so that
        \[ f(\mathcal{S}_k^*) \le  \frac{2k-1}{k}f(\mathcal{S}_k) +k\theta_a.\]
Dividing by $\frac{k}{2k-1}=1-\frac{k-1}{2k-1}$ gives
\begin{align*}
f(\mathcal{S}_k^*)\leq f(\mathcal{S}_k) +\frac{k-1}{2k-1}f(\mathcal{S}_k^*)+\frac{k^2}{2k-1}\theta_a.
\end{align*}

\paragraph{Multiplicative relaxation} 
    The interchange property (\ref{e_mult}) ensures
    \begin{equation*}
    f(\mathcal{S}_{k-1} \cup \{s_j^*\}) \le \theta_m  f(\mathcal{S}_k), \qquad 1 \le j \le k.
    \end{equation*}
Subtract $f(\mathcal{S}_{k-1})$ on both sides, and use the expression for $\zeta_k$,
    \[ f(\mathcal{S}_{k-1} \cup \{s_j^*\} ) - f(\mathcal{S}_{k-1}) \le  \theta_m f(\mathcal{S}_k)  - f(\mathcal{S}_{k-1}) = \zeta_{k} + (\theta_m-1) f(\mathcal{S}_{k})  , \qquad 1\le j \le k.   \]
    Insert this inequality into~\eqref{eqn:decomp2},
    \begin{align*} f(\mathcal{S}_k^*) &\le   f(\mathcal{S}_{k-1})    +k (\theta_m-1) f(\mathcal{T}_{k}) + k \zeta_{k}   
        \leq f(\mathcal{S}_k)\left( 1 + k(\theta_m-1) + \frac{k-1}{k} \right)\\
        &= \frac{f(\mathcal{S}_k)}{k}\left(2k-1+k^2(\theta_m-1)\right),
    \end{align*}
where the last inequality follows from~\eqref{eqn:decomp}.
    \end{proof}

Theorem~\ref{t_53} implies that the solution from the 1-interchange Algorithm~\ref{alg_proto2} tends to differ more
from the optimal solution, as the  
relaxation factors $\theta_a$ and $\theta_m$ and 
the desired number $k$ of columns increase.

\subsection{Existing Work}\label{s_2rel}
 A survey of set submodular functions, their applications and extensions is given in \cite{Krause2014}.
Distributed greedy algorithms for maximizing set submodular functions are analyzed in \cite{Mirza2013}.
In the context of determinantal point processes, there are greedy algorithms for set log-submodular functions \cite{Gillen12}
and lazy greedy algorithms~\cite{Hemmi22}.
Extensions to non-monotone set submodular functions are analyzed in \cite{Bodek22,Buch25,Buch14,Sakaue20}.
Adaptive submodularity is introduced and surveyed in  \cite{GoK10,GoK11}.
Randomized greedy algorithms for sensor selection are presented in \cite{Hash2020}.

%% file: sec3.tex
\section{Submodular functions in subset selection}\label{s_3}
After defining the volume of a matrix (Definition~\ref{d_vol}), we review work on column subset selection based on
volume maximization (Section~\ref{s_rel3}), and show that the log-volume is a set submodular function (Section~\ref{s_lv}).

Unlike \cite[Definition 1.1]{Ben1992} and \cite[Section 2]{MikOs18},
the definition below is strict in that it allows a positive volume
only for matrices with non-zero singular values.

\begin{definition}\label{d_vol}
The volume of a matrix $\mx\in\real^{m\times n}$ with $\rank(\mx)=n$ equals
\begin{equation*}
\vol(\mx)\equiv\sqrt{\det(\mx^T\mx)}=
\prod_{i=1}^n{\sigma_i(\mx)}.
\end{equation*}
The volume of a matrix $\mx\in\real^{m\times n}$ with $\rank(\mx)=m$ equals
\begin{equation*}
\vol(\mx)\equiv\sqrt{\det(\mx\mx^T)}=\prod_{i=1}^m{\sigma_i(\mx)}.
\end{equation*}
The volume of a nonsingular matrix $\mx\in\rmm$ equals
\begin{equation*}
\vol(\mx)=|\det(\mx)|=\prod_{i=1}^m{\sigma_i(\mx)}.
\end{equation*}
%
\end{definition}

\subsection{Existing work}\label{s_rel3}
Volume maximization is NP-hard.
Specifically, given $\mx\in\rmn$ with $\rank(\mx)\geq k$, finding a $m\times k$ column submatrix of $\mx$
with maximal volume is a NP-hard problem \cite[Theorem 6]{CM09}, \cite{Papa84}.
Moreover \cite[Section 1]{DEFM15}, given $\mx\in\real^{k\times n}$ with $\rank(\mx)=k$,
there exists a universal constant such $c>0$ so that
finding a $k\times k$ maximal-volume submatrix of $\ma$  within a factor $c^k$ is also NP-hard.

The 
Businger-Golub algorithm 
\cite{BusGo65}, \cite[Section 5.4.2]{GoVL13}
computes the pivoted QR factorization of a matrix
$\mx\in\rmn$, and it
corresponds to Khachiyan's algorithm for finding the maximum subdeterminant \cite[Figure 1]{DEFM15}. 
If $\mx_k^*\in\real^{m\times k}$ is a submatrix
of $\mx$ with maximal volume, then the output
$\mx_k\in\real^{m\times k}$ from the Businger-Golub algorithm
satisfies \cite[Theorem 11]{CM09}
\begin{equation*}
\vol(\mx_k)\geq \frac{1}{k!} \vol(\mx_k^*).
\end{equation*}

Bounds for the volume can also be derived from bounds for individual singular values. 

\begin{lemma}\label{lem:busgolub}
    Let $\mx \in \rmn$ and $k$ a fixed integer with $0<k \le \rank(\mx)$. 
   Let $\mx_k^*\in\real^{m\times k}$ be a column submatrix of $\mx$ that maximizes
$\log{\vol(\mb)}$ among all submatrices $\mb\in\real^{m\times k}$ of $\mx$.
Then the output
$\mx_k\in\real^{m\times k}$ from the Businger-Golub algorithm
satisfies 
\[ \frac{\vol(\mx_k^*)}{\prod_{j=1}^k 2^j \sqrt{n-j}} \le \vol(\mx_k) \le \vol(\mx_k^*).\] 
\end{lemma}
\begin{proof}
The upper bound follows from the definition of $\mx_k^*$. The bounds for individual singular values from \cite[Theorem 7.2]{GuEis96} imply
\[ \frac{\sigma_j(\mx)}{2^j \sqrt{n-j}} \le \sigma_j(\mx_k) \le \sigma_j(\mx), \qquad 1 \le j \le k. \]  
Multiplication gives
\[ \prod_{j=1}^k  \frac{\sigma_j(\mx)}{2^j\sqrt{n-j}}   \le 
{\prod_{j=1}^k\sigma_j(\mx_k)} = {\vol(\mx_k)} .\]
Singular value interlacing~\eqref{e_inter} implies $\vol
(\mx_k^*) = \prod_{j=1}^k  \sigma_j(\mx_k^*) \le \prod_{j=1}^k  \sigma_j(\mx)$. Insert this
into the previous inequality to obtain the lower bound.
\end{proof}

Additional bounds for the case $ m\le k \le n$ are given in~\cite[Section 2.1]{osinsky2026subsetselectionmatricescolumn}. 

Next are bounds from the Gu-Eisenstat strong rank revealing QR algorithm~\cite{GuEis96}.

\begin{lemma}
    Let $\mx \in \rmn$ and $k$ a fixed integer with $0<k \le \rank(\mx)$.
   Let $\mx_k^*\in\real^{m\times k}$ be a column submatrix of $\mx$ that maximizes
$\log{\vol(\mb)}$ among all submatrices $\mb\in\real^{m\times k}$ of $\mx$.
 Then the output
$\mx_k\in\real^{m\times k}$ from the Gu-Eisenstat~\cite[Algorithm 3]{GuEis96} with relaxation factor $f = \eta$ satisfies 
\[   \frac{\vol(\mx_k^*)}{(1+\eta^2k(n-k))^{k/2}}   \le 
{\vol(\mx_k)} \le \vol(\mx_k^*).\]
\end{lemma}

\begin{proof}
The proof is similar to that of Lemma~\ref{lem:busgolub} but uses~\cite[Theorem 3.2]{GuEis96} instead. 
\end{proof}

Algorithms for volume maximization are presented in \cite{Osinsky23,osinsky2026subsetselectionmatricescolumn} when $k$ exceeds the number of rows,
with bounds on the norm of the pseudo-inverse.

The notion of submodularity for subset selection has also been explored in the context of regression~\cite{DasK2018},
graph Laplacian column selection~\cite{fornace2024column}, and minimization of the matrix residual
$\min_{\mc}\|\ma - \mc\mc^\dagger\ma\|_F$~\cite{altschuler2016greedy}.
The submodularity of the log-volume is discussed in~\cite{kelmans1983multiplicative,taskar2013determinantal}.

\subsection{Log-volume}\label{s_lv}
After presenting two examples (Examples \ref{ex_1} and~\ref{ex_2}), and expressing set submodular functions
in terms of matrices (Remark~\ref{r_1}),
we show that the log-volume is a submodular function
(Theorem~\ref{t_logvol}) and establish conditions
when it is non-decreasing (Lemma~\ref{l_1}).

\begin{remark}[Set submodular functions in terms of matrices]\label{r_1}
Given a matrix $\mx\in\rmn$. With the notation in
Definition~\ref{d_set}, let
$\mathcal{S}\subset\mathcal{T}\subset\{1,\ldots, n\}$,
so that $\mx_{\mathcal{S}}$ is a submatrix of $\mx_{\mathcal{T}}$
which, in turn, is a submatrix of $\mx$. Let 
$t\in\{1,\ldots,n\}-\mathcal{T}$.

A function $f$ on column subsets of~$\mx$ is 
\textit{set submodular}
if for all submatrices
\begin{equation*}
f(\mx_{\mathcal{S}\cup\{t\}})-f(\mx_{\mathcal{S}})\geq 
f(\mx_{\mathcal{T}\cup\{t\}})-f(\mx_{\mathcal{T}}).
\end{equation*}
That is, adding a column $\mx_{\{t\}}$ to a smaller submatrix $\mx_{\mathcal{S}}$ is more beneficial than adding 
$\mx_{\{t\}}$ to a larger submatrix $\mx_{\mathcal{T}}$.
\end{remark}

\begin{example}\label{ex_1}
The volume itself is \emph{not} a set submodular function on column submatrices.

With the notation in Definition~\ref{d_set}, let
\begin{equation*}
\mx=\begin{bmatrix} 2 & 0 & 0&1\\ 0&3 &0 &0\\ 0&0&4&0\end{bmatrix}\in\real^{3\times 4},
\end{equation*}
$\mathcal{S}=\{1\}$, and $\mathcal{T}=\{1,2\}$.
Then 
\begin{align*}
\vol(\mx_{\mathcal{S}})=2, \quad 
\vol(\mx_{\mathcal{S}\cup\{3\}})=8,\quad 
\vol(\mx_{\mathcal{T}})=6, \quad 
\vol(\mx_{\mathcal{T}\cup\{3\}})=24,
\end{align*}
and
\begin{equation*}
\vol(\mx_{\mathcal{S}\cup\{3\}})-\vol(\mx_{\mathcal{S}})=6\not\geq 18= 
\vol(\mx_{\mathcal{T}\cup\{3\}})-\vol(\mx_{\mathcal{T}}).
\end{equation*}
However, since the matrices are diagonal,
\begin{equation*}
\log{\vol(\mx_{\mathcal{S}\cup\{3\}})}-\log{\vol(\mx_{\mathcal{S}})}=\log(4)=
\log{\vol(\mx_{\mathcal{T}\cup\{3\}})}-
\log{\vol(\mx_{\mathcal{T}})}.
\end{equation*}
\end{example}

The log-volume of diagonal matrices is easy to determine.

\begin{example}\label{ex_2}
Let $\mx\in\rmn$ be a diagonal matrix with $\rank(\mx)=n$. Then $\log\vol(\mx)$
is a \emph{modular} function on column submatrices of $\mx$.

A set submodular function that satisfies the condition in Definition~\ref{d_1} with equality is called \emph{modular} \cite[Section 1.1]{Mirza2013},
\begin{equation*}
f(\mathcal{T}\cup \{x\}) - f(\mathcal{T})= f(\mathcal{S}\cup\{x\})-f(\mathcal{S})
\end{equation*}
for all $\mathcal{S}, \mathcal{T}\subset\mathcal{X}$ and 
$x\in\mathcal{X}-(\mathcal{S}\cup\mathcal{T})$.
\end{example}

Although the submodularity of the log-volume was proved
in
\cite{kelmans1983multiplicative}, \cite[Section 2.4.5]{taskar2013determinantal} and \cite[Appendix]{Sham2010}
we present a clean proof for general matrices below, starting with
a well known result about determinants.

\begin{lemma}\label{l_det}
Let $\mx\in\rmn$, $\mathcal{S}\subset\{1,\ldots,n\}$,
$t\in\{1,\ldots,n\} -\mathcal{S}$, and
$\mP_{\mathcal{S}}\equiv \mx_{\mathcal{S}}(\mx_{\mathcal{S}}^T\mx_{\mathcal{S}})^{\dagger}\mx_{\mathcal{S}}$ the 
orthogonal projector onto $\range(\mx_{\mathcal{S}})$.
Then
\begin{align*}
 \vol(\mx_{\mathcal{S}\cup\{t\}}) = \vol(\mx_{\mathcal{S}})\| (\mi-\mP_{\mathcal{S}} )\mx_{\{t\}}\|.
\end{align*}
\end{lemma}

\begin{proof}
Applying Schur's determinant formula for the 
generalized Schur complement \cite[Section 1.6]{horn2005basic} to
\begin{align*}
\mx_{\mathcal{S}\cup\{t\}}^T\mx_{\mathcal{S}\cup\{t\}}=
\begin{bmatrix}\mx_{\mathcal{S}}^T\mx_{\mathcal{S}}&
\mx_{\mathcal{S}}^T\mx_{\{t\}}\\
\mx_{\{t\}}^T\mx_{\mathcal{S}} &\mx_{\{t\}}^T\mx_{\{t\}}
\end{bmatrix}
\end{align*}
gives
\begin{align*}
\det\left(\mx_{\mathcal{S}\cup\{t\}}^T\mx_{\mathcal{S}\cup\{t\}}\right)=
\det\left(\mx_{\mathcal{S}}^T\mx_{\mathcal{S}}\right)
\det\left(\mx_{\{t\}}^T(\mi-\mP_{\mathcal{S}})\mx_{\{t\}}\right).
\end{align*}
Hence
$\vol(\mx_{\mathcal{S}\cup\{t\}}) = \vol(\mx_{\mathcal{S}})\| (\mi-\mP_{\mathcal{S}} )\mx_{\{t\}}\|$.
\end{proof}

\begin{theorem}\label{t_logvol}
For  $\mx\in\rmn$, $\log\vol(\mx)$
is a set submodular function on column submatrices of~$\mx$.
\end{theorem}

\begin{proof}
Since $\log{\vol(\mx)}=\tfrac{1}{2}\log{\det(\mx^T\mx)}$,
it suffices to show the submodularity 
of $\log{\det(\mx^T\mx)}$. 

Let $\mathcal{S}\subset\mathcal{T}\subset\{1,\ldots,n\}$, $t\in\{1,\ldots,n\}-\mathcal{T}$, and define orthogonal projectors
$\mP_{\mathcal{S}}\equiv \mx_{\mathcal{S}}(\mx_{\mathcal{S}}^T\mx_{\mathcal{S}})^{\dagger}\mx_{\mathcal{S}}^T$ and
$\mP_{\mathcal{T}}\equiv \mx_{\mathcal{T}}(\mx_{\mathcal{T}}^T\mx_{\mathcal{T}})^{\dagger}\mx_{\mathcal{T}}^T$. Lemma~\ref{l_det}
implies
\begin{align*}
 \vol(\mx_{\mathcal{S}\cup\{t\}}) &= \vol(\mx_{\mathcal{S}})\| (\mi-\mP_{\mathcal{S}} )\mx_{\{t\}}\|\\
 \vol(\mx_{\mathcal{T}\cup\{t\}}) & = \vol(\mx_{\mathcal{T}})\| (\mi-\mP_{\mathcal{T}} )\mx_{\{t\}}\|.
\end{align*}
The Loewner inequality $\mP_{\mathcal{S}} \preceq \mP_{\mathcal{T}}$ implies 
\begin{equation*}
\| (\mi-\mP_{\mathcal{S}} )\mx_{\{t\}}\|^2=
\mx_{\{t\}}^T (\mi - \mP_{\mathcal{S}})\mx_{\{t\} }
\ge\mx_{\{t\}}^T (\mi - \mP_{\mathcal{T}})\mx_{\{t\}}=
\| (\mi-\mP_{\mathcal{T}} )\mx_{\{t\}}\|^2.
\end{equation*}
Combining everything gives
\begin{align*}
 \vol(\mx_{\mathcal{S}\cup\{t\}}) \geq\vol(\mx_{\mathcal{S}})\| (\mi-\mP_{\mathcal{T}})
 \mx_{\{t\}}\|
 =\vol(\mx_{\mathcal{S}})\vol(\mx_{\mathcal{T}\cup\{t\}})/
 \vol(\mx_{\mathcal{T}}).
\end{align*}
At  last, take logarithms and interpret $0/0$ as $0$.
\end{proof}

For full column-rank matrices with all singular values greater or equal to one, log-volume is a non-negative non-decreasing function.

\begin{lemma}\label{l_1} 
Let $\mx\in\real^{m\times n}$ have $\rank(\mx)=n$
and singular values $\sigma_1(\mx)\geq \cdots \geq \sigma_n(\mx)\geq 1$. Then  
$\log\vol(\mx)\geq 0$.

{Furthermore, if $\mathcal{S} \subset \mathcal{T}\subset\{1,\ldots,n\}$, then $0 \le \log\vol(\mx_{\mathcal{S}}) \le \log\vol(\mx_{\mathcal{T}})$.}
\end{lemma}

\begin{proof}
From $\det(\mx^T\mx)=\prod_{i=1}^n{\sigma_i(\mx)^2}\geq 1$ follows $\log{\vol(\mx)}\geq 0$.

Singular value interlacing (\ref{e_so}) implies $1\leq \vol(\mx_{\mathcal{S}})\leq \vol(\mx_{\mathcal{T}})$. Therefore
$0\leq\log{\vol(\mx_{\mathcal{S}})}\leq \log{\vol(\mx_{\mathcal{T}})}$.
\end{proof}

%% file: sec4.tex
\section{A greedy algorithm for column subset selection}\label{s_4}
We show that the Businger-Golub QR factorization 
(Algorithm~\ref{alg_qrcp})
is a greedy algorithm in the sense of Algorithm~\ref{alg_proto}
for maximizing the log-volume of a full column-rank matrix (Section~\ref{s_BG1}),
and extend this to general matrices via a  smoothed analysis (Section~\ref{s_BG3}).

\begin{algorithm}[!t]
\caption{Conceptual version of Businger-Golub QR \cite{BusGo65,GoVL13}}\label{alg_qrcp}

\begin{algorithmic}[1]
\REQUIRE $\mx\in\rmn$, $0<k\leq\rank(\mx)$

\STATE $\mP_0\equiv\vzero_{m\times m}$, $\mathcal{S}_0=\emptyset$,
$\mathcal{T}_0=\{1,\ldots, n\}$
\FOR{$j=1:k$}
\STATE \COMMENT{Find projected column with largest norm}
\STATE Select column $\vx_t$ of $\mx_{\mathcal{T}_{j-1}}$ with 
$\|(\mi-\mP_{j-1})\vx_{t}\|\equiv\max_{i\in\mathcal{T}_{j-1}}{\|(\mi-\mP_{j-1})\vx_i\|}$
\smallskip

\STATE $\mathcal{S}_j=\mathcal{S}_{j-1} \cup \{t\}$, 
$\mathcal{T}_j=\mathcal{T}_{j-1}-\{t\}\qquad$
\COMMENT{$\mathcal{S}_j\cup\mathcal{T}_j=\{1,\ldots,n\}$}
\STATE $\vy_j\equiv (\mi-\mP_{j-1})\vx_t$,
$\mP_j=\mP_{j-1}+\vy_j\vy_j^{\dagger}\quad$ \COMMENT{Accumulate projections}
\ENDFOR
\bigskip

\RETURN $m\times k$ column submatrix $\mx_{\mathcal{S}_k}$ 
\end{algorithmic}
\end{algorithm}

Algorithm~\ref{alg_qrcp} views the Businger-Golub QR factorization
\cite{BusGo65}, \cite[Section 5.4.2]{GoVL13} from
a modified Gram-Schmidt perspective,
with a focus on selecting $k$ columns rather than computing
a QR factorization.

\subsection{Log-volume for full column-rank matrices}\label{s_BG1}
We show that Algorithm~\ref{alg_qrcp}
is a greedy algorithm 
for maximizing the log-volume of a full column-rank matrix with
sufficiently large singular values (Theorem~\ref{t_2}).
For full-column matrices with general singular singular values,
we give bounds for scaled matrices (Corollary~\ref{c_2}), 
and discuss the ramifications of the scaling (Remark~\ref{r_2}).

Theorem~\ref{t_2} below  shows that Algorithm~\ref{alg_qrcp}
is a greedy algorithm and it
approximates the maximal log-volume with an error of at most 37 percent when
applied to full column-rank matrices with sufficiently
large singular values. 

\begin{theorem}\label{t_2}
Let  $\mx\in\rmn$ have $\rank(\mx)=n$ and singular values $\sigma_1(\mx)\geq \cdots \geq\sigma_n(\mx)\geq 1$.
Let  $k$ be a fixed integer with  $0<k<n$. Let $\mx_k^*\in\real^{m\times k}$ be a column submatrix of $\mx$ that maximizes
$\log{\vol(\mb)}$ among all submatrices $\mb\in\real^{m\times k}$ of $\mx$.

Algorithm~\ref{alg_qrcp} is a greedy algorithm in the sense
of Algorithm~\ref{alg_proto} for maximizing the log-volume, and its output matrix 
$\mx_k\equiv\mx_{\mathcal{S}_k}$ satisfies
\begin{equation}\label{e_t2a}
0\leq\log{\vol(\mx_k^*)}-\log{\vol(\mx_k)}\leq (1-\tfrac{1}{k})^k\ \log{\vol(\mx_k^*)}
\end{equation}
where $(1-\tfrac{1}{k})^k\leq \tfrac{1}{e}< .37$.

The upper bound in (\ref{e_t2a}) is tight for matrices $\mx$ with orthonormal columns.
\end{theorem}

\begin{proof}
Theorem~\ref{t_logvol} and Lemma~\ref{l_1} imply that log-volume is a non-negative,
non-decreasing set submodular function on column submatrices of
$\mx\in\rmn$. This justifies that 
we can initialize the volume of the empty matrix
as $\vol(\emptyset)=1$, so that $\log{\vol(\emptyset)}=0$.

Algorithm~\ref{alg_qrcp} selects $k$ columns $\mx_{\mathcal{S}_k}$
of $\mx$. Let $\mx_{\mathcal{S}_k}=\mq\mr$ be a QR decomposition 
where $\mq\in\real^{m\times k}$ has orthonormal columns and $\mr\in\real^{k\times k}$ is nonsingular,
with non-increasing diagonal elements
and $r_{11}\geq \ldots \geq r_{kk}> 0$.

First we show that Algorithm~\ref{alg_qrcp} is a greedy algorithm in the sense of
Algorithm~\ref{alg_proto}. 

\paragraph{Induction basis}
For $j=1$, let  $\mx_{\mathcal{S}_1}\in\real^m$ be a 
column of $\mx$ with QR decomposition is
$\mx_{\mathcal{S}_1}=\mq_1r_{11}$ 
where $\mq_1\in\real^m$ has $\|\mq_1\|=1$ and $r_{11}\geq 0$.
Algorithm~\ref{alg_qrcp} selects the column with largest two-norm, so that among all columns of $\mx$,
\begin{equation*}
r_{11}^2=\|\mx_{\mathcal{S}_1}\|^2=
\det(\mx_{\mathcal{S}_1}^T\mx_{\mathcal{S}_1})=
\vol(\mx_{\mathcal{S}_1})^2
\end{equation*}
is maximal. Let
$\mP_1\equiv \mx_{\mathcal{S}_1}\mx_{\mathcal{S}_1}^{\dagger}=\mq_1\mq_1^T$ be the  orthogonal projector onto
$\range(\mx_{\mathcal{S}_{1}})$.
\smallskip

\paragraph{Induction hypothesis} Assume that for $1\leq j-1<n$, Algorithm~\ref{alg_qrcp} is a greedy algorithm and has selected the $j-1$ linearly 
independent columns 
$\mx_{\mathcal{S}_{j-1}}=\mq_{j-1} \mr_{j-1}$
where $\mq_{j-1}\in\real^{m\times (j-1)}$ 
has orthonormal columns, $\mr_{j-1}\in\real^{(j-1)\times (j-1)}$ is nonsingular upper triangular, and 
\begin{equation*}
r_{11}^2 \cdots r_{j-1,j-1}^2=\det(\mr_{j-1}^T\mr_{j-1})=
\det(\mx_{\mathcal{S}_{j-1}}^T\mx_{\mathcal{S}_{j-1}})=
\vol(\mx_{\mathcal{S}_{j-1}})^2.
\end{equation*}
Let
$\mP_{j-1}\equiv\mx_{\mathcal{S}_{j-1}}\mx_{\mathcal{S}_{j-1}}^{\dagger}=\mq_{j-1}\mq_{j-1}^T$
be the orthogonal projector onto $\range(\mx_{\mathcal{S}_{j-1}})$.
\smallskip

\paragraph{Induction step}
Let $\mx_{\mathcal{S}_j}\equiv\begin{bmatrix} \mx_{\mathcal{S}_{j-1}} & \vx_t\end{bmatrix}=\mq_j\mr_j$ where
\begin{equation*}
\mq_j=\begin{bmatrix} \mq_{j-1} & \vq_j\end{bmatrix}\in\real^{m\times j}, \qquad
\mr_j\equiv \begin{bmatrix} \mr_{j-1} & *\\ \vzero & r_{jj}\end{bmatrix}\in\real^{j\times j},
\end{equation*}
$\mq_j$ has orthonormal columns, and $\vq_jr_{jj}=(\mi-\mP_{j-1})\vx_j$.
Among all columns $i\in\mathcal{T}_{j-1}$, Algorithm~\ref{alg_qrcp} 
selects the column $\vx_t$ with maximal 
$r_{jj}=\|(\mi-\mP_{j-1})\vx_t\|$.
With $\vy_j\equiv (\mi-\mP_{j-1})\vx_t$, set 
\begin{equation*}
\mP_j\equiv \mP_{j-1}+\vy_j\vy_j^{\dagger}=\mq_{j-1}\mq_{j-1}^T
+\vq_j\vq_j^T=\mq_j\mq_j^T=\mx_{\mathcal{S}_j}\mx_{\mathcal{S}_j}^{\dagger}.
\end{equation*}
Because $\mr_j$ is square,
\begin{align*}
\det(\mx_{\mathcal{S}_j}^T\mx_{\mathcal{S}_j})&=\det(\mr_j^T\mr_j)=\det(\mr_j)^2
=\det(\mr_{j-1})^2\>r_{jj}^2
=\det(\mx_{\mathcal{S}_{j-1}}^T\mx_{\mathcal{S}_{j-1}})\> 
r_{jj}^2.
\end{align*}
Since Algorithm~\ref{alg_qrcp} chooses a column $\vx_t$
with largest $r_{jj}$, it is  a greedy algorithm
in the sense of Algorithm~\ref{alg_proto}.
Now apply Theorem~\ref{t_ng} to the log-volume.

The upper bound in~(\ref{e_t2a}) is tight for matrices $\mx$ with orthonormal columns, since $\log{\vol(\mb)}=0$ for any
submatrix $\mb\in\real^{m\times k}$ of $\mx$.
\end{proof}

For $k=1$, $\mx_k=\mx_k^*$ and the bound (\ref{e_t2a}) is zero.

If $\mx$ has singular values less than one,
we multiply by a scalar to increase all singular values sufficiently.
This kind of manipulation is not always recommended \cite[Section 1]{Corne1977} as it makes the relative error not unique. 
However, since Algorithm~\ref{alg_qrcp} is invariant under scalar multiplication, we feel that this is justified.

\begin{corollary}\label{c_2}
Let  $\mx\in\rmn$ have $\rank(\mx)=n$,
and let  $k$ be a fixed integer with  $0<k<n$.
Let $\mx_k^*\in\real^{m\times k}$ be a column submatrix of $\mx$ that maximizes
$\log{\vol(\mb)}$ among all submatrices $\mb\in\real^{m\times k}$ of $\mx$.
The output matrix $\mx_k\equiv\mx_{\mathcal{S}_k}$ from Algorithm~\ref{alg_qrcp} satisfies
\begin{equation*}
0\leq \log{\vol(\mx_k^*)}-\log{\vol(\mx_k)}\leq (1-\tfrac{1}{k})^k\ \log{\vol(\mx_k^*/\sigma_n(\mx))},
\end{equation*}
where $(1-\tfrac{1}{k})^k\leq \tfrac{1}{e}< .37$.
\end{corollary}

\begin{proof}
To scale matrices so that the log-volume is a non-negative
non-decreasing function,
let the singular values $\sigma_1(\mx)\geq \cdots \geq \sigma_n(\mx)\geq\tfrac{1}{\alpha}>0$ for some $\alpha>0$.
Since $\sigma_1(\alpha\mx)\geq\cdots\geq \sigma_n(\alpha\mx)\geq 1$, Lemma~\ref{l_1} implies that
$\log\vol(\alpha\mx)\geq 0$ and non-decreasing, that is:
If $\mathcal{S}\subset\mathcal{T}\subset\{1,\ldots, n\}$, then
$0\leq\log{\vol(\alpha\mx_{\mathcal{S}})}\leq \log{\vol(\alpha\mx_{\mathcal{T}})}$.

Furthermore, Algorithm~\ref{alg_qrcp} is invariant under multiplication by a scalar $\alpha>0$
and so is the absolute error,
\begin{align*}
\log{\vol(\alpha\mx_k^*)}-\log{\vol(\alpha\mx_k)}&=(k\log{\alpha}+\log{\vol(\mx_k^*)})-(k\log{\alpha}+\log{\vol(\mx_k)})\\
 &=  \log{\vol(\mx_k^*)}-\log{\vol(\mx_k)}.
\end{align*}
Choose $\alpha=1/\sigma_n(\mx)$. This ensures that $\log{\vol(\mb)}\geq 0$ for any submatrix $\mb\in\real^{m\times k}$ 
of $\mx$. At last apply Theorem~\ref{t_2} to the scaled
matrix $\mx/\sigma_n(\mx)$.
\end{proof}

\begin{remark}\label{r_2}
If all singular values of $\mx$ are the same, i.e. $\sigma_1(\mx)=\sigma_n(\mx)$, then any submatrix $\mb\in\real^{m\times k}$ of $\mx$
maximizes the volume, and satisfies
$\log{\vol(\mx_k^*)}=\log{\vol(\mb)}$.
Thus, Algorithm~\ref{alg_qrcp} terminates with an optimal solution in $k$ steps. The scaling in Corollary~\ref{c_2}
ensures a zero error bound, with $\log{\vol(\mx_k^*/\sigma_n(\mx))}=0$,
and thus the tightness of the bound.

If $\sigma_n(\mx)<1$, the scaling in Corollary~\ref{c_2} guarantees that the log-volume of the scaled matrix is non-negative.
If $\sigma_n(\mx)>1$, then the scaling lowers the bound on the absolute error.
\end{remark}

\subsection{Log-volume for general matrices}\label{s_BG3}
General, possibly rank deficient matrices require a different approach because Theorem~\ref{t_2} and Corollary~\ref{c_2} apply 
only to full-rank matrices.
We apply smoothed analysis, which perturbs the matrix by a small amount with a Gaussian random matrix, to square matrices
(Theorem~\ref{thm:smoothed}) and then discuss extensions to general matrices (Remark~\ref{r_logvol_gen}).

Perturb the square matrix $\mx \in \real^{n\times n}$ by a standard Gaussian random matrix\footnote{The elements of a standard Gaussian random matrix are independent Gaussian random variables with zero mean and variance~1.} $\mg$  to $\mz \equiv \mx + \epsilon\mg$ where $\epsilon > 0$.
Then $\mz$ 
is nonsingular with high probability
\cite{sankar2006smoothed}, so that we can apply Corollary~\ref{c_2} to $\mz$.

Theorem~\ref{thm:smoothed} shows that for sufficiently small perturbation $\epsilon$, the conclusions of Corollary~\ref{c_2} hold with a few modifications. The condition on $\epsilon$ ensures that the selected columns $\mx_k$ have full rank with high probability, since the corresponding columns of $\mz_k$ do. 

\begin{theorem}\label{thm:smoothed}
Let $\mx \in \rnn$ and let $k$ be a fixed integer
with $0 < k \leq \rank(\mx)$. Let $\mx_k^*\in\real^{n\times k}$
be a column submatrix of $\mx$ that maximizes
$\log{\vol(\mb)}$ among all submatrices $\mb\in\real^{n\times k}$ of $\mx$. 

Define the perturbation $\mz\equiv \mx+\epsilon\mg$,
where $\epsilon>0$ and $\mg$ is a standard Gaussian random matrix.   Let $\mz_k^*\in\real^{n\times k}$
be a column submatrix of $\mz$ that maximizes
$\log{\vol(\mb)}$ among all submatrices $\mb\in\real^{n\times k}$ of $\mz$. Let $\mz_k$ be the output from Algorithm~\ref{alg_qrcp}, and $\mx_k$ the corresponding columns of $\mx_k$.

Let $0 < \delta < 1$ represent the failure probability and define
\[ \Delta_1 \equiv \frac{\delta\epsilon}{4.7\sqrt{n}}, \qquad \Delta_2 \equiv 2\sqrt{n}  + \sqrt{2\log(4/\delta)}.  \] 

If
$$0 < \epsilon < \min\left\{\frac{\sigma_{k}(\mx_k^*)}{2\Delta_2}, \frac{\sigma_k(\mx)}{\Delta_2(1 + 2^{k+1}\sqrt{n-k})} \right\}$$ then, with probability at least $1-\delta$, 
\[\begin{aligned} 0 \le \log\vol(\mx_k^*)- \log\vol(\mx_k) \le & \>  \frac1e \log\vol(\mx_k^*/\Delta_1)  + k(2\eta_2 + \eta_1),
\end{aligned}\]
where $\eta_1 \equiv \epsilon\Delta_2/\sigma_{k}(\mz_k) \le 1/2$ and $\eta_2 \equiv \epsilon\Delta_2/\sigma_{k}(\mx_k^*) \le 1/2$. 
\end{theorem}

\begin{proof}
    See Appendix~\ref{app_a}.
\end{proof}

Compared to Corollary~\ref{c_2}, the absolute error in Theorem~\ref{thm:smoothed} has an additional term bounded by $3k/2$ and the smallest singular value is replaced by $\Delta_1$.
This could
possibly make Theorem~\ref{thm:smoothed} pessimistic in practice.

Theorem~\ref{thm:smoothed} relies on a bound for the smallest singular value of $\mz$, which only holds for square matrices~\cite{sankar2006smoothed}. Remark~\ref{r_logvol_gen}
proposes extensions to rectangular matrices. 

\begin{remark}\label{r_logvol_gen}
Consider the cases $m>n$ and $m<n$.
\begin{itemize}
\item If $m > n$, compute a thin-QR factorization $\mx = \mq\bmat{\mr \\ \boldsymbol{0}}$. Since the columns of $\mr \in \real^{n\times n}$ have the same volume as the corresponding columns of $\mx$, we can apply Theorem~\ref{thm:smoothed} to $\mr$ instead.
\item If $m < n$, apply Theorem~\ref{thm:smoothed} to the padded matrix $\mx_{\rm pad} = \bmat{\mx \\ \boldsymbol{0}} \in \real^{n\times n}$, whose columns have the same volume as the corresponding columns of $\mx$.
\end{itemize}
\end{remark}

%% file: sec5.tex
\section{A 1-interchange algorithm for column subset selection}\label{sec:geqr}
We show that the Gu-Eisenstat QR algorithm 
(Algorithm~\ref{alg_gue}) is a 1-interchange
algorithm in the sense of Algorithm~\ref{alg_proto2} for maximizing
the log-volume of a full column-rank matrix
(Section~\ref{s_GE1}). A smoothed analysis extends this to general matrices (Section~\ref{s_GE_23}).

\begin{algorithm}[!t]
\caption{Conceptual version of Gu-Eisenstat QR \cite{GuEis96}}\label{alg_gue}
 \begin{algorithmic}[1]
\REQUIRE $\mx\in\rmn$, $0<k\leq\rank(\mx)$, relaxation factor
$\eta\geq 1$
\medskip

\STATE \COMMENT{Initialization: Select $k$ linearly independent columns of $\mx$}
\STATE Find submatrix $\mx_{\mathcal{S}}\in\real^{m\times k}$ of $\mx$ with $\rank(\mx_{\mathcal{S}})=k$
\smallskip

\REPEAT
\STATE \COMMENT{Swap column $i\in\mathcal{S}$  with column $j\in\{1,\ldots, n\}-\mathcal{S}$}
\STATE $\mathcal{T}=(\mathcal{S}-\{i\})\cup\{j\}$
\medskip

\STATE \COMMENT{Did the swap increase the volume?}
\IF{$\vol(\mx_{\mathcal{T}})>\eta \vol(\mx_{\mathcal{S}})$}
\STATE $\mathcal{S}=\mathcal{T}\qquad$ \COMMENT{Keep the changes}
\ENDIF
\UNTIL{no more increase in $\vol(\mx_{\mathcal{S}})$}
\bigskip

\RETURN $m\times k$ column submatrix $\mx_{\mathcal{S}}$ 
\end{algorithmic}
\end{algorithm}

Algorithm~\ref{alg_gue} represents a conceptual version of the Gu-Eisenstat QR factorization
\cite[Algorithm 5]{GuEis96}, with a focus
on selecting $k$ columns rather than computing a QR factorization.
Algorithm~\ref{alg_gue}
starts with an initial set of $k$ columns that are linearly independent.
Without this initialization, the first swap could still produce a rank deficient submatrix 
with zero volume, thus the algorithm would make no progress and terminate. 
The relaxation factor $\eta=1$ represents an ideal version, where Algorithm~\ref{alg_gue} is continued until the volume of the 
leading $k$ columns stops increasing.

\subsection{Log-volume for full column-rank matrices}\label{s_GE1}
We show that Algorithm~\ref{alg_gue} 
is a 1-interchange algorithm for maximizing the log-volume of a full column-rank matrix in two ways:
in the sense of Algorithm~\ref{alg_proto2}
if the relaxation factor $\eta=1$
(Theorem~\ref{t_3} and Corollary~\ref{c_3}); and in the 
sense of Theorem~\ref{t_53} if the relaxation factor
$\eta>1$ (Theorem~\ref{c_54})
in which case we can bound the number of swaps 
(Theorem~\ref{t_55}).

Theorem~\ref{t_3} below shows that
Algorithm~\ref{alg_gue} with relaxation factor $\eta=1$
applied to a full column-rank
matrix with sufficiently large singular
values
is an instance of the $R$-interchange algorithm with $R=1$ \cite[Section 5]{Nemhauser1978}, that
approximates the maximal log-volume with an error of at most 50 percent.

\begin{theorem}\label{t_3}
Let  $\mx\in\rmn$ have $\rank(\mx)=n$ with singular
values $\sigma_1(\mx)\geq \cdots \geq \sigma_n(\mx)\geq 1$, and
let  $k$ be a fixed integer with  $0<k< n$.
Let $\mx_k^*\in\real^{m\times k}$ be a column submatrix of $\mx$ that maximizes
$\log{\vol(\mb)}$ among all submatrices $\mb\in\real^{m\times k}$ of $\mx$. 

Algorithm~\ref{alg_gue} with relaxation factor $\eta=1$ 
is a 1-interchange algorithm in the sense of Algorithm~\ref{alg_proto2} for maximizing the log-volume, and its output matrix $\mx_k\equiv \mx_{\mathcal{S}}$
satisfies
\begin{equation}\label{e_t3a}
0\leq \log{\vol(\mx_k^*)}-\log{\vol(\mx_k)}\leq \frac{k-1}{2k-1}\ \log{\vol(\mx_k^*)},
\end{equation}
where $\frac{k-1}{2k-1}<\frac{1}{2}$. 

The upper bound in~(\ref{e_t3a}) is tight for matrices $\mx$ with orthonormal columns. 
\end{theorem}

\begin{proof}
This follows from Theorem~\ref{t_ng2} applied to the function $f(\mathcal{S}) = \log\vol(\mx_{\mathcal{S}})$ and Lemma~\ref{l_1}.
\end{proof}

For $k=1$, $\mx_k=\mx_k^*$ and the bound (\ref{e_t3a})
is equal to zero.

Corollary~\ref{c_3} presents a bound, based on scaling, for full column-rank matrices 
with general singular values, and the observations from Remark~\ref{r_2} apply here as well.

\begin{corollary}\label{c_3}
Let  $\mx\in\rmn$ with $\rank(\mx)=n$, and
let  $k$ be a fixed integer with  $0<k<n$.
Let $\mx_k^*\in\real^{m\times k}$ be a column submatrix of~$\mx$ that maximizes
$\log{\vol(\mb)}$ among all submatrices $\mb\in\real^{m\times k}$ of $\mx$. The output
matrix $\mx_k\equiv\mx_{\mathcal{S}}$ from Algorithm~\ref{alg_gue} with relaxation
factor $\eta=1$
satisfies
\begin{equation*}
0\leq \log{\vol(\mx_k^*)}-\log{\vol(\mx_k)}\leq 
\frac{k-1}{2k-1}\ \log{\vol(\mx_k^*/\sigma_n(\mx))},
\end{equation*}
where $\frac{k-1}{2k-1}<\frac{1}{2}$.
\end{corollary}

\begin{proof}
This
follows from Theorem~\ref{t_3} applied to the scaled matrix $\mx/\sigma_n(\mx)$.
\end{proof}

Theorem~\ref{c_54} shows that Algorithm~\ref{alg_gue}
with relaxation
factor $\eta > 1$ is a 1-interchange algorithm for maximizing the log-volume,
and implies that the larger~$\eta$, the
larger the error in the log-volume.

\begin{theorem}\label{c_54}
Let  $\mx\in\rmn$ have $\rank(\mx)=n$,
and let  $k$ be a fixed integer with  $0<k<n$.
Let $\mx_k^*\in\real^{m\times k}$ be a column submatrix of $\mx$ that maximizes
$\log{\vol(\mb)}$ among all submatrices $\mb\in\real^{m\times k}$ of $\mx$. 

Algorithm~\ref{alg_gue} with relaxation factor $\eta>1$
is a 1-interchange algorithm in the sense of Theorem~\ref{t_53}, and its output
matrix~$\mx_k\equiv\mx_{\mathcal{S}}$ satisfies
\begin{equation*}
0\leq \log{\vol(\mx_k^*)}-\log{\vol(\mx_k)}
\leq \frac{k-1}{2k-1}\log\vol(\mx_k^*/\sigma_n(\mx)) + \frac{k^2}{2k-1}\,\log(\eta),
\end{equation*}
where $\frac{k-1}{2k-1}<\frac{1}{2}$ and
$\frac{k^2}{2k-1}\leq k$.
\end{theorem}

\begin{proof}
This follows from Theorem~\ref{t_53} with additive
relaxation factor $\theta_a = \log \eta$.
\end{proof}

An initialization of Algorithm~\ref{alg_gue} with Algorithm~\ref{alg_qrcp}
is recommended in \cite[Section 5.2]{damle2025estimating},
to reduce the number of swaps.
However \cite[Section 5]{Nemhauser1978} cautions that
a greedy initialization may not improve worst case behavior.
Nevertheless, Theorem~\ref{t_55} below bounds the number of swaps.

\begin{theorem}\label{t_55}
Let  $\mx\in\rmn$ have $\rank(\mx)=n$,
and let  $k$ be a fixed integer with  $0<k<n$.
Let $\mx_k^*\in\real^{m\times k}$ be a column submatrix of $\mx$ that maximizes
$\log{\vol(\mb)}$ among all submatrices $\mb\in\real^{m\times k}$ of $\mx$.

If line~1 of Algorithm~\ref{alg_gue} with 
relaxation factor $\eta>1$ is initialized with 
the output matrix $\mx_{\mathcal{S}_k}$ from Algorithm~\ref{alg_qrcp}, the maximal number
of swaps is
\begin{equation*}\label{eqn:gueisbound}
q \le \frac{\log(\vol(\mx_k^*/\sigma_n(\mx))}{e\,\log(\eta)} \le 
 \frac{k}{e}\log_\eta(\kappa_2(\mx)).
 \end{equation*}
\end{theorem}

\begin{proof}
Let $\mx_k$ be the output
from Algorithm~\ref{alg_qrcp}. Corollary~\ref{c_2}
implies
\begin{align*}
\log\left(\vol(\mx_k^*)/\vol(\mx_k)\right)\leq
\log\left(\vol(\mx_k^*)^{1/e}/\sigma_n(\mx)^{k/e}\right)
\end{align*}
Since the logarithm is a monotone function,
\[ {\vol(\mx_k^*)^\gamma}{\sigma_n(\mx)^{k/e}} \le \vol(\mx_k) \le \vol(\mx_k^*), \qquad \gamma \equiv 1-1/e. 
\]
Each swap of Algorithm~\ref{alg_gue} improves the volume by a factor $\eta > 1$. After $q$ swaps, 
the improvement is $\eta^q$, hence
\[\eta^q \, {\vol(\mx_k^*)^\gamma}{\sigma_n(\mx)^{k/e}}
\le \eta^q\,\vol(\mx_k)\leq \vol(\mx_k^*). \]
Rearranging the lower and upper bounds gives
\[ \eta^q \le \vol(\mx_k^*)^{1/e} /\sigma_n(\mx)^{k/e} = \vol(\mx_k^*/\sigma_n(\mx))^{1/e}.   \] 
 Thus, the maximal number of swaps is 
 \begin{equation*}
 q \le \frac{1}{e}\frac{\log(\vol(\mx_k^*/\sigma_n(\mx))}{\log(\eta)} \le 
 \frac{k}{e}\log_\eta(\kappa_2(\mx)).
 \end{equation*}$\qquad$
\end{proof}

Theorem~\ref{t_55} implies that the number of swaps 
decreases as the relaxation factor increases.
 In contrast to 
the bound $q \le k \log_\eta(\sqrt{n})$ from~\cite[Section 4.4]{GuEis96} and~\cite[Section 5.2]{damle2025estimating},
the bound in Theorem~\ref{t_55} is independent of $n$ and tighter if $ \kappa_2(\mx)^{1/e} < \sqrt{n}$; that is, if the matrix is well-conditioned. Combining the bounds 
shows that the number of swaps cannot exceed
\begin{equation}
q \le k \min\{ \log_\eta(\sqrt{n}), \log_\eta(\kappa_2(\mx))/e\}.
\end{equation}

\subsection{Log-volume for general matrices}\label{s_GE_23}
As in Section~\ref{s_BG3}, we apply Algorithm~\ref{alg_gue}
to the perturbed matrix $\mz = \mx + \epsilon \mg$ where $\mg$ is a standard Gaussian random matrix and $\epsilon>0$. As a consequence, Corollary~\ref{c_3} applies to general square matrices 
(Theorem~\ref{thm:smoothed2}).

\begin{theorem}\label{thm:smoothed2}
Let $\mx \in \rnn$ and let $k$ be a fixed integer
with $0 < k \leq \rank(\mx)$. Let $\mx_k^*\in\real^{n\times k}$
be a column submatrix of $\mx$ that maximizes
$\log{\vol(\mb)}$ among all submatrices $\mb\in\real^{n\times k}$ of $\mx$. 

Define the perturbation $\mz\equiv \mx+\epsilon\mg$,
where $\epsilon>0$ and $\mg$ is a standard Gaussian random matrix.   Let $\mz_k^*\in\real^{n\times k}$
be a column submatrix of $\mz$ that maximizes
$\log{\vol(\mb)}$ among all submatrices $\mb\in\real^{n\times k}$ of $\mz$. Let $\mz_k$ be the output from Algorithm~\ref{alg_gue}, and $\mx_k$ the corresponding columns of $\mx_k$.

Let $0 < \delta < 1$ represent the failure probability and define
\[ \Delta_1 \equiv \frac{\delta\epsilon}{4.7\sqrt{n}}, \qquad \Delta_2 \equiv 2\sqrt{n}  + \sqrt{2\log(4/\delta)}.  \] 
If
$$0 < \epsilon < \min\left\{\frac{\sigma_{k}(\mx_k^*)}{2\Delta_2}, \frac{\sigma_k(\mx)}{\Delta_2(1 + 2\sqrt{1 + k(n-k)})} \right\},$$ then, with probability at least $1-\delta$, 
\[\begin{aligned} 0 \le \log\vol(\mx_k^*)- \log\vol(\mx_k) \le & \>  \frac12 \log\vol(\mx_k^*/\Delta_1)  + k(2\eta_2 + \eta_1),
\end{aligned}\]
where $\eta_1 \equiv\epsilon\Delta_2/\sigma_{k}(\mz_k) \le 1/2$ and $\eta_2 \equiv \epsilon\Delta_2/\sigma_{k}(\mx_k^*) \le 1/2$. 
\end{theorem}

\begin{proof}
See Appendix~\ref{app_c}.
\end{proof}

Theorem~\ref{thm:smoothed2} can be extended to rectangular matrices as shown in Remark~\ref{r_logvol_gen}.

%% file: sec6.tex
\section{Greedy and 1-interchange algorithms for submatrix selection}\label{s_6}
We analyze algorithms for selecting principal submatrices 
of maximal volume from a symmetric positive-definite matrix.

After reviewing existing work (Section~\ref{s_rel6}), we show that 
pivoted Cholesky (Algorithm~\ref{alg_pchol}) 
is a greedy algorithm 
in the sense of Algorithm~\ref{alg_proto}
for maximizing the log-volume of 
positive definite matrices (Section~\ref{ssec:greedychol})
and of positive semidefinite matrices (Section~\ref{ssec:logvolspsd}).
Then we show
 that strong rank-revealing Cholesky
(Algorithm~\ref{alg_srrchol})
 is a 1-interchange algorithm in the sense of Algorithm~\ref{alg_proto2} for maximizing the log-volume of positive semidefinite matrices  (Section~\ref{s_63}).

\subsection{Existing work}\label{s_rel6}
Finding a $k\times k$ submatrix with maximal volume
$\ma\in\rnn$ is NP-hard \cite[Theorem 6]{CM09}, \cite[Section 1]{Massei2022}.

For matrices $\ma\in\rnn$ with $\rank(\ma)\geq k$, a $k\times k$ submatrix with largest volume 
is attained by a principal submatrix if $\ma$ is diagonally dominant \cite[Theorem 4]{CKM2020}, or if
$\ma$ is symmetric positive semi-definite \cite[Theorem 1]{CKM2020}.
Low-rank approximations have been implemented by pivoted Cholesky factorizations~\cite{Harb2012}, and randomized pivoted Cholesky factorizations~\cite{chen2025randomly,epperly2025embrace}.

The pivoted Cholesky algorithm, in Algorithm~\ref{alg_pchol}, coincides  with adaptive cross approximation \cite[Algorithms 1 and 2]{Beb2000} when applied to symmetric positive semi-definite matrices \cite[Section 1]{Harb2012}.
ACA computes a low-rank approximation, but 
can be viewed as a greedy method for volume maximization, and amounts to Gaussian elimination process with rook
pivoting~\cite[Section 3]{Beb2000}, \cite[Section 1]{Massei2022}.
It seems to be the same
as the fast Greedy MAP inference \cite[Algorithm 1]{ChenZhang2018}. Other algorithms for volume maximization have been proposed in~\cite{goreinov2010find,damle2025estimating}. 

Log-volume maximization via pivoted Cholesky is implemented as fast Greedy MAP inference \cite[Section 3]{ChenZhang2018}.

\begin{algorithm}[!t]
\caption{Conceptual version of pivoted Cholesky \cite{GoVL13,Higham2002}}\label{alg_pchol}
\begin{algorithmic}[1]
\REQUIRE Symmetric positive semi-definite $\ma\in\rnn$, $0<k\leq\rank(\ma)$

\STATE $\ma_0\equiv \ma$, $\mathcal{S}_0=\emptyset$,
$\mathcal{T}_0\equiv \{1,\ldots, n\}$
\FOR{$j=1:k$}
\STATE \COMMENT{Select largest diagonal element}
\STATE Choose column $\va_t$ of $\ma_{j-1}$ with
$a_{tt}=\max_{i\in\mathcal{T}_{j-1}}{a_{ii}}$
\smallskip

\STATE $\ma_j=\ma_{j-1} -\va_t a_{tt}^{-1}\va_t^T\qquad$
\COMMENT{Downdate}
\smallskip

\STATE $\mathcal{S}_j=\mathcal{S}_{j-1}\cup\{t\}$,
$\mathcal{T}_j=\mathcal{T}_{j-1}-\{t\}\qquad$
\COMMENT{$\mathcal{S}_j\cup\mathcal{T}_j=\{1,\ldots,n\}$}
\ENDFOR
\bigskip

\RETURN $k\times k$ principal submatrix $\ma[\mathcal{S}_k]$
\end{algorithmic}
\end{algorithm}

\subsection{Greedy algorithmfor positive definite matrices}\label{ssec:greedychol}
After relating Algorithm~\ref{alg_pchol} to Algorithm~\ref{alg_qrcp} (Theorem~\ref{thm:pivcholqrcp}),
we show that Algorithm~\ref{alg_pchol} is a greedy algorithm for maximizing the log-volume
of symmetric positive definite matrices
(Theorem~\ref{thm:pivcholesky}).

Algorithm~\ref{alg_pchol} presents the conceptual
version of the pivoted Cholesky 
algorithm \cite[Section 4.2.7]{GoVL13}, \cite[Section 10.3]{Higham2002}, with a focus on selecting a $k\times k$ principal
submatrix rather than computing a Cholesky factorization.

Theorem~\ref{thm:pivcholqrcp} below shows that Algorithm~\ref{alg_qrcp} applied to a Cholesky factor $\mx$ of~$\ma$ implements Algorithm~\ref{alg_pchol}.
The Cholesky factor $\mx$ does not have to be a triangular matrix.

\begin{theorem}\label{thm:pivcholqrcp}
Let $\ma=\mx^T\mx$ where $\mx\in\rnn$, and let $k$
be a fixed integer with $0<k\leq \rank(\mx)$.
Let $\mx_{\mathcal{S}_k}\in\real^{n\times k}$ be the 
output matrix of Algorithm~\ref{alg_qrcp}. Then the output
matrix
of Algorithm~\ref{alg_pchol} satisfies
\begin{equation*}
\ma[\mathcal{S}_k]=\mx_{\mathcal{S}_k}^T\mx_{\mathcal{S}_k}.
\end{equation*}
\end{theorem}

\begin{proof}
With an induction proof over the 
dimension of the principal submatrices, we show that Algorithm~\ref{alg_qrcp} selects the same indices,
apart from ties,
as Algorithm~\ref{alg_pchol}.

\paragraph{Induction basis}
For $j=1$, Algorithm~\ref{alg_qrcp} selects a column
$\mx_{\mathcal{S}_1}$ of $\mx$ 
with largest two-norm, so that among all columns of $\mx$,
\begin{equation*}
\|\mx_{\mathcal{S}_1}\|^2=
\det(\mx_{\mathcal{S}_1}^T\mx_{\mathcal{S}_1})=
\vol(\ma[\mathcal{S}_1])
\end{equation*}
is maximal. The corresponding column and diagonal
element of $\ma=\mx^T\mx$ 
are $\va_t\equiv \mx^T\mx_{\mathcal{S}_1}$,
and
$a_{tt}=\mx_{\mathcal{S}_1}^T\mx_{\mathcal{S}_1}$,
respectively.
Then
\begin{align*}
\mx^T(\mi-\mP_1)\mx=
\ma-\mx^T\mx_{\mathcal{S}_1}(\mx_{\mathcal{S}_1}^T\mx_{\mathcal{S}_1})^{-1}\mx_{\mathcal{S}_1}^T\mx
=\ma-\va_t a_{tt}^{-1} \va_t^T=\ma_1,
\end{align*}
where $\mP_1\equiv \mx_{\mathcal{S}_1}\mx_{\mathcal{S}_1}^{\dagger}$ is the orthogonal projector onto 
$\range(\mx_{\mathcal{S}_1})$.
With $\mx=\begin{bmatrix}\mx_{\mathcal{S}_1}&\mx_{\mathcal{T}_1}\end{bmatrix}$ this gives
\begin{equation*}
\ma_1=\begin{bmatrix} 0 & \vzero \\ \vzero &
\mx_{\mathcal{T}_1}^T(\mi-\mP_1)\mx_{\mathcal{T}_1}.
\end{bmatrix}
\end{equation*}

\paragraph{Induction hypothesis} Assume that for $1\leq j-1<n$, Algorithm~\ref{alg_qrcp} has selected the $j-1$ linearly 
independent columns $\mx_{\mathcal{S}_{j-1}}$, and
\begin{equation*}
\det(\mx_{\mathcal{S}_{j-1}}^T\mx_{\mathcal{S}_{j-1}})=
\vol(\ma[\mathcal{S}_{j-1}]).
\end{equation*}
Let
$\mP_{j-1}\equiv\mx_{\mathcal{S}_{j-1}}\mx_{\mathcal{S}_{j-1}}^{\dagger}=\mq_{j-1}\mq_{j-1}^T$
be the orthogonal projector onto $\range(\mx_{\mathcal{S}_{j-1}})$. Then
\begin{align*}
\mx^T(\mi-\mP_{j-1})\mx&=
\ma-\mx^T\mx_{\mathcal{S}_{j-1}}(\mx_{\mathcal{S}_{j-1}}^T\mx_{\mathcal{S}_{j-1}})^{-1}\mx_{\mathcal{S}_{j-1}}^T\mx\\
&=\begin{bmatrix} \vzero_{(j-1)\times (j-1)} & \vzero \\ \vzero &
\mx_{\mathcal{T}_{j-1}}^T(\mi-\mP_{j-1})\mx_{\mathcal{T}_{j-1}}\end{bmatrix}=\ma_{j-1}.
\end{align*}
\smallskip

\paragraph{Induction step}
Let $\mx_{\mathcal{S}_j}\equiv\begin{bmatrix} \mx_{\mathcal{S}_{j-1}} & \vx_t\end{bmatrix}$, 
where 
Algorithm~\ref{alg_qrcp} 
selects a column $\vx_t\in\mx_{\mathcal{T}_{j-1}}$ with maximal 
$\|(\mi-\mP_{j-1})\vx_t\|$.
With $\vy_j\equiv (\mi-\mP_{j-1})\vx_t$, 
the corresponding column $\va_t$ and diagonal element
$a_{tt}$ of $\ma_{j-1}$ are respectively
\begin{equation*}
\va_t=\begin{bmatrix}\vzero_{(j-1)\times 1}\\ \mx_{\mathcal{T}_{j-1}}^T(\mi-\mP_{j-1})\vy_j\end{bmatrix}\quad \text{and}
\quad a_{tt}=\vy_j^T\vy_j.
\end{equation*}
Update $\mx_{\mathcal{T}_{j-1}}\equiv\begin{bmatrix} \vx_t& \mx_{\mathcal{T}_{j}} \end{bmatrix}$
and abbreviate
\begin{equation*}
(\mi-\mP_{j-1})\mx_{\mathcal{T}_{j-1}}=
\begin{bmatrix} \vy_j & \my_j\end{bmatrix}\qquad\text{where}\quad
\my_j\equiv (\mi-\mP_{j-1})\mx_{\mathcal{T}_j}.
\end{equation*}
Then the downdate is
\begin{equation*}
\ma_j=\ma_{j-1}-\va_t a_{tt}^{-1}\va^T=
\begin{bmatrix} \vzero_{j\times j} & \vzero \\ \vzero & \mz_j\end{bmatrix}
\end{equation*}
where
\begin{align*}
\mz_j&\equiv \my_j^T\my_j-\my_j^T\vy_j(\vy_j^T\vy_j)^{-1}\vy_j^T\my_j=
\my_j^T\left(\mi-\vy_j\vy_j^{\dagger}\right)\my_j
=\mx_{\mathcal{T}_j}^T(\mi-\mP_j)\mx_{\mathcal{T}_j}
\end{align*}
and
$\mP_j\equiv \mP_{j-1}+\vy_j\vy_j^{\dagger}
=\mx_{\mathcal{S}_j}\mx_{\mathcal{S}_j}^{\dagger}$
is the orthogonal projector onto $\range(\mx_{\mathcal{S}_j})$.
Thus,
\begin{equation*}
\ma_j=\begin{bmatrix} \vzero_{j\times j} & \vzero \\ \vzero &
\mx_{\mathcal{T}_{j}}^T(\mi-\mP_{j})\mx_{\mathcal{T}_{j}}\end{bmatrix}.
\end{equation*}
The proof of Theorem~\ref{t_2} implies
\begin{align*}
\vol(\ma[\mathcal{S}_j])=\det(\mx_{\mathcal{S}_j}^T\mx_{\mathcal{S}_j})&
=\det(\mx_{\mathcal{S}_{j-1}}^T\mx_{\mathcal{S}_{j-1}})\> 
\|\vy_j\|^2=\vol[(\ma[\mathcal{S}_{j-1}])\|\vy_j\|^2.
\end{align*}
\end{proof}

Theorem~\ref{thm:pivcholesky} below
shows that Algorithm~\ref{alg_pchol} 
is a greedy algorithm for maximizing the log-volume
of symmetric positive definite matrices.

\begin{theorem}\label{thm:pivcholesky}
Let $\ma\in\rnn$ be symmetric positive definite,
and let $k$ be a fixed integer with $0<k<n$.
Let $\ma_k^*\in\real^{k\times k}$ be a principal submatrix of~$\ma$ that maximizes $\log\vol(\mb)$ among all principal submatrices 
$\mb\in\real^{k\times k}$ of $\ma$. 

Then Algorithm~\ref{alg_pchol} is a greedy algorithm in the
sense of Algorithm~\ref{alg_proto} for maximizing 
the log-volume and its output matrix
$\ma_k\equiv\ma[\mathcal{S}_k]$ satisfies
\begin{equation}\label{e_piva}
0 \le {\log\vol(\ma_k^*)-\log\vol(\ma_k)} \le \left(1-\frac1k\right)^k \log\vol(\ma_k^*/\lambda_n(\ma)),
    \end{equation}
    where $(1-\tfrac{1}{k})^k\leq \tfrac{1}{e} <.37$.

The upper bound in (\ref{e_piva}) is tight for matrices $\ma$ that are non-zero multiples of the identity matrix. 
\end{theorem}

\begin{proof}
This follows from Theorem~\ref{thm:pivcholqrcp} and Corollary~\ref{c_2}.
\end{proof}

Symmetric positive definite matrices $\ma$ arise  in applications, such as  sensor placement for Gaussian processes,
where $\ma = \mi + \sigma_n^{-2}\mk$, with $\mk$ being a positive semidefinite kernel matrix and $\sigma_n$ the noise standard deviation~\cite{chen2026optimal,mirzasoleiman2015lazier,Mirza2013}. In such applications,  Theorem~\ref{thm:pivcholesky} applies with $\lambda_n(\ma) \ge 1$.

\subsection{Greedy algorithm for positive semidefinite matrices}\label{ssec:logvolspsd}
As in Section~\ref{s_BG3}, we extend the error bounds for
the greedy algorithm to positive semidefinite matrices (Theorem~\ref{thm:smoothed_psd})

To this end, apply
Algorithm~\ref{alg_pchol} to the perturbed matrix $\mh \equiv \ma + \epsilon\mg\mg^T$, where $\epsilon>0$ and
$\mg \in \real^{n\times n}$ is a standard Gaussian matrix, since
$\mh$ is positive definite with high probability. Although Algorithm~\ref{alg_pchol} outputs the principal submatrix $\mh_k$,
the following theorem derives bounds on the log-volume of the corresponding submatrix~$\ma_k$ of $\ma$.

\begin{theorem}\label{thm:smoothed_psd}
Let $\ma\in\rnn$ be symmetric positive semidefinite,
and let $k$ be a fixed integer with $0<k\le \rank(\ma)$.
Let $\ma_k^*\in\real^{k\times k}$ be a principal submatrix of~$\ma$ that maximizes $\log\vol(\mb)$ among all principal submatrices 
$\mb\in\real^{k\times k}$ of $\ma$. 

Define the perturbation $\mh \equiv \ma + \epsilon \mg\mg^T$, where $\mg \in \real^{n\times n}$ is a standard Gaussian random matrix and $\epsilon>0$. Let $\mh_k = \mh[\mathcal{S}_k]$ be 
a submatrix of $\mh$ selected by Algorithm~\ref{alg_pchol} and let $\ma_k\equiv \ma[\mathcal{S}_k]$ be the corresponding submatrix of $\ma$.

Let $0 < \delta < 1$ represent the failure probability. Define 
\[ \Delta_1 \equiv \epsilon\left(\frac{\delta}{2n}\right)^2 , \qquad \Delta_2 = (2\sqrt{n} + \sqrt{2\log(4/\delta)})^2. \]
If
$$0 < \epsilon < \min\left\{\frac{\lambda_{k}(\ma_k^*)}{2\Delta_2}, \frac{\lambda_k(\ma)}{\Delta_2(1 + 2^{k+2}{n-k})} \right\},$$ 
then, with probability at least $1-\delta$, 
\[\begin{aligned} 0 \le \log\vol(\ma_k^*)- \log\vol(\ma_k) \le & \>  \frac1e \log\vol(\ma_k^*/\Delta_1)  + k(2\eta_2 + \eta_1),
\end{aligned}\]
where $\eta_1 \equiv \epsilon\Delta_2/\lambda_{k}(\mh_k) \le 1/2$ and $\eta_2 \equiv \epsilon\Delta_2/\lambda_{k}(\ma_k^*) \le 1/2$. 
\end{theorem}

\begin{proof}
    See Appendix~\ref{app_b}.
\end{proof}

\subsection{Interchange algorithm for positive semidefinite matrices}\label{s_63}
After relating Algorithm~\ref{alg_srrchol} to Algorithm~\ref{alg_gue} (Theorem~\ref{t_100}),
we show that Algorithm~\ref{alg_srrchol} is a 1-interchange algorithm for maximizing the log-volume
of a symmetric positive definite matrix in two ways: in the sense of Algorithm~\ref{alg_proto2}
if the relaxation factor $\eta=1$
(Corollary~\ref{c_100});
and in the sense of Theorem~\ref{t_53}
if the relaxation factor $\eta>1$
(Theorem~\ref{thm:srrcholesky}). Then we extend the latter
to positive semidefinite matrices
(Theorem~\ref{thm:smoothed_psd2}).

A conceptual version of \cite[Algorithm 1]{gu2004strong} is presented in Algorithm~\ref{alg_srrchol}, with a
focus on selecting a $k\times k$ principal submatrix rather than
computing a Cholesky factorization.

Theorem~\ref{t_100} below shows that Algorithm~\ref{alg_gue} applied to a Cholesky factor $\mx$ of~$\ma$ implements Algorithm~\ref{alg_srrchol}.
The Cholesky factor $\mx$ does not have to be a triangular matrix.

\begin{theorem}\label{t_100}
Let $\ma=\mx^T\mx$ where $\mx\in\rnn$, and let $k$
be a fixed integer with $0<k\leq \rank(\mx)$.
Let $\mx_{\mathcal{S}_k}\in\real^{n\times k}$ be the 
output matrix of Algorithm~\ref{alg_gue}. Then the output
matrix
of Algorithm~\ref{alg_srrchol} satisfies
\begin{equation*}
\ma[\mathcal{S}_k]=\mx_{\mathcal{S}_k}^T\mx_{\mathcal{S}_k}.
\end{equation*}
\end{theorem}

\begin{proof}
This is analogous to the proof of Theorem~\ref{thm:pivcholqrcp}.
\end{proof}

\begin{algorithm}[!t]
\caption{Conceptual version of strong rank-revealing Cholesky \cite{gu2004strong}} \label{alg_srrchol}
 \begin{algorithmic}[1]
\REQUIRE Symmetric positive semi-definite matrix $\ma\in\rnn$, $0<k\leq\rank(\ma)$, \\
$\qquad$ Relaxation factor $\eta \geq 1$
\medskip

\STATE \COMMENT{Initialization}
\STATE Find nonsingular principal submatrix
$\ma[\mathcal{S}]\in\real^{k\times k}$ of $\ma$
\smallskip

\REPEAT
\STATE \COMMENT{Swap column $i\in\mathcal{S}$
with column $j\in\{1,\ldots,n\}-\mathcal{S}$}
\STATE $\mathcal{T}=\mathcal{S}-\{i\}\cup\{j\}$
\medskip

\STATE \COMMENT{Did swap increase the volume?}
\smallskip

\IF{$\vol(\ma[\mathcal{T}])> \sqrt{\eta} \,\vol(\ma[\mathcal{S}])$}
\STATE $\mathcal{S}=\mathcal{T}\qquad$ \COMMENT{Keep the changes}
\ENDIF
\UNTIL{no more increase in $\vol(\ma[\mathcal{S}])$}
\bigskip

\RETURN $k\times k$ principal submatrix $\ma[\mathcal{S}]$
\end{algorithmic}
\end{algorithm}

Corollary~\ref{c_100} below shows that Algorithm~\ref{alg_srrchol} with relaxation factor $\eta=1$ is a 1-interchange algorithm for maximizing the log-volume.

\begin{corollary}\label{c_100}
Let  $\ma\in\rnn$ be symmetric positive definite, and
let  $k$ be a fixed integer with  $0<k<n$.
Let $\ma_k^*\in\real^{m\times k}$ be a 
principal submatrix of~$\ma$ that maximizes
$\log{\vol(\mb)}$ among all principal submatrices $\mb\in\real^{k\times k}$ of $\ma$. 

Then Algorithm~\ref{alg_srrchol} with relaxation
factor $\eta=1$ is a 1-interchange algorithm
in the sense of Algorithm~\ref{alg_proto2} for maximizing
the log-volume, and its output
matrix $\ma_k\equiv\ma[\mathcal{S}]$
satisfies
\begin{equation*}
0\leq \log{\vol(\ma_k^*)}-\log{\vol(\ma_k)}\leq 
\frac{k-1}{2k-1}\ \log{\vol(\ma_k^*/\lambda_n(\ma))},
\end{equation*}
where $\frac{k-1}{2k-1}<\frac{1}{2}$.
\end{corollary}

\begin{proof}
This follows from Theorem \ref{t_100} and
Corollary~\ref{c_3}.
\end{proof}

Theorem~\ref{thm:srrcholesky} below shows that Algorithm~\ref{alg_srrchol}
with relaxation
factor $\eta > 1$ is a 1-interchange algorithm for maximizing the log-volume,
and implies that the larger~$\eta$, the
larger the error in the log-volume.

\begin{theorem}\label{thm:srrcholesky}
Let $\ma\in\rnn$ be symmetric positive definite, and
let $0<k<n$ be a fixed integer.
Let $\ma_k^*\in\real^{k\times k}$ be a principal submatrix of $\ma$ that maximizes $\log\vol(\mb)$ among all principal submatrices $\mb\in\real^{k\times k}$ of $\ma$. 
    
Then Algorithm~\ref{alg_srrchol} with relaxation factor $\eta>1$ is a 1-interchange algorithm in the sense of Theorem~\ref{t_53} for maximizing the log-volume, and its output matrix $\ma_k\equiv\ma[\mathcal{S}]$ satisfies 
    \[ 0 \le \log\vol(\ma_k^*)-\log\vol(\ma_k) \le  \frac{k}{2k-1}\log\vol(\ma_k^*/\lambda_n(\ma)) + \frac{k^2}{2k-1}\log \sqrt{\eta}, \]
where $\frac{k-1}{2k-1}<\frac{1}{2}$ and
$\frac{k^2}{2k-1}\leq k$. 
\end{theorem}

\begin{proof}
This follows from Theorems \ref{t_100} and~\ref{c_54}.
\end{proof}

As in Section~\ref{ssec:logvolspsd}, we can extend Theorem~\ref{thm:srrcholesky} to symmetric positive semidefinite matrices.
\begin{theorem}\label{thm:smoothed_psd2}
Let $\ma\in\rnn$ be symmetric positive semidefinite,
and let $k$ be a fixed integer with $0<k\le \rank(\ma)$.
Let $\ma_k^*\in\real^{k\times k}$ be a principal submatrix of~$\ma$ that maximizes $\log\vol(\mb)$ among all principal submatrices 
$\mb\in\real^{k\times k}$ of $\ma$. 

Define the perturbation $\mh\equiv \ma+\epsilon\mg\mg^T$,
where $\epsilon>0$ and $\mg\in\rnn$ is a standard Gaussian random matrix.   Let $\mh_k \equiv \mh[\mathcal{S}]$ be the output from Algorithm~\ref{alg_srrchol}, and $\ma_k\equiv \ma[\mathcal{S}]$ be the corresponding submatrix.

Let $0 < \delta < 1$ represent the failure probability and define
\[ \Delta_1 \equiv \epsilon\left(\frac{\delta}{2n}\right)^2 , \qquad \Delta_2 = (2\sqrt{n} + \sqrt{2\log(4/\delta)})^2. \]
If
$$0 < \epsilon < \min\left\{\frac{\lambda_{k}(\ma_k^*)}{2\Delta_2}, \frac{\lambda_k(\ma)}{\Delta_2(3+ 2k(n-k))} \right\},$$ 
then, with probability at least $1-\delta$, 
\[\begin{aligned} 0 \le \log\vol(\ma_k^*)- \log\vol(\ma_k) \le & \>  \frac12 \log\vol(\ma_k^*/\Delta_1)  + k(2\eta_2 + \eta_1),
\end{aligned}\]
where $\eta_1 \equiv \epsilon\Delta_2/\lambda_{k}(\mh_k) \le 1/2$ and $\eta_2 \equiv \epsilon\Delta_2/\lambda_{k}(\ma_k^*) \le 1/2$. 
\end{theorem}

\begin{proof}
See Appendix~\ref{app_d}.   
\end{proof}


%% file: sec7.tex
\section{Possible directions for future work}\label{s_conc}

We have established connections between column subset selection and set submodular functions, and shown that log-volume is a 
set submodular function on the columns of a matrix.

Avenues for future work
include connections to submodular optimization for algorithm development. Since log-volume 
is in general not monotone, one 
could look towards greedy algorithms for non-monotone functions in~\cite{Buch25,Buch14}. 
There are other types of algorithms for set submodular functions such as lazy greedy, and stochastic greedy~\cite{mirzasoleiman2015lazier}. The distributed greedy method in~\cite{Mirza2013} appears to be related to the notion of tournament pivoting in rank-revealing QR factorizations~\cite{demmel2015communication}. There may also be opportunities for randomized QR factorizations such as~\cite{duersch2017randomized} and~\cite{grigori2025randomized}. 
Another option is to relax submodularity and 
and replace it by approximate submodularity \cite{DasK2018}
instead. Extending the notion of set submodularity to other algorithms such as local maximum volume submatrix~\cite[Algorithms 1 and 2]{damle2025estimating} and criteria such as S-opt~\cite{SX2016,LCS2024} are also of interest. 
Finally, the more restrictive notion of \textit{string submodularity} \cite{Street2007,Zhang2016} may be applicable to other greedy
algorithms for column subset selection.

%% file: sec8.tex
\appendix\label{s_app}
\section{Proof of Theorem~\ref{thm:smoothed}}
Before proving Theorem~\ref{thm:smoothed} (Section~\ref{app_a}),
Theorem~\ref{thm:smoothed2} (Section~\ref{app_c}),
Theorem~\ref{thm:smoothed_psd} (Section~\ref{app_b})
and Theorem~\ref{thm:smoothed_psd2} (Section~\ref{app_d}), we
present three auxiliary results for: log volume 
perturbations (Lem\--ma~\ref{lemma:logvolperturb}), 
small singular values 
of matrices perturbed by a Gaussian (Theorem~\ref{thm:gaussthm1}), 
and norms of Gaussians (Theorem~\ref{thm:gaussthm2}).

\begin{lemma}\label{lemma:logvolperturb}
    Let $\mc, \md \in \real^{m\times k}$ with $\md \equiv \mc + \me$  and $\rank(\mc) = k$. If $\tau \equiv \|\me\|_2 / \sigma_k(\mc)\le 1/2$, then $\rank(\md) =k$ and 
    \[    -k\tau + \log\vol(\mc) \le \log\vol(\md)  \le   \log\vol(\mc) +  k\tau  .\]
   \end{lemma}

\begin{proof}
    Weyl's inequality~\eqref{eqn:weyl} implies
    \[ \sigma_j(\mc) - \|\me\|_2 \le \sigma_j(\md) \le \sigma_j(\mc) + \|\me\|_2, \qquad 1 \le j \le k. \]
From  $\|\me\|_2 \le \sigma_k(\mc)/2$ follows $\sigma_k(\md) \ge  \sigma_k(\mc)/2 > 0$, hence $\rank(\md) =k$. Factoring
    \[ \sigma_j(\mc)\left(1 - \frac{\|\me\|_2}{\sigma_j(\mc)} \right) \le \sigma_j(\md) \le  \sigma_j(\mc)\left(1 +  \frac{\|\me\|_2}{\sigma_j(\mc)} \right), \qquad 1 \le j\le k \] 
gives
    \[ 1 - \tau \le \frac{\sigma_j(\md)}{\sigma_j(\mc)} \le 1 + \tau, \qquad 1 \le j \le k.\] 
After taking logarithms and summing, use the fact that $\log(1+x)\le |x|$ for $|x| \le \frac12$.  
\end{proof}

\begin{theorem}\label{thm:gaussthm1} 
Let $\mx \in \real^{n\times n}$ be a fixed matrix and $\mz \equiv \mx + \epsilon \mg$, where $\mg$ is a standard Gaussian random matrix and $\epsilon>0$. Then, for $0 < \delta < 1$,  
\[ \mathbb{P}\left\{ \sigma_n(\mz) \le \frac{\epsilon \delta}{2.35\sqrt{n}} \right\} \le \delta.   \] 
\end{theorem}

\begin{proof}
    From~\cite[Theorem 3.3]{sankar2006smoothed} follows for $t>0$,
    \[ \mathbb{P}\left\{ \|\mz^{-1}\|_2 \ge t\right\} \le \frac{2.35\sqrt{n}}{\epsilon t}.  \] 
    Set the right hand side equal to $\delta$, solve for $t$, and write $\|\mz^{-1}\|_2 = 1/\sigma_n(\mz)$. 
\end{proof}

\begin{theorem}\label{thm:gaussthm2} Let $\mg\in \real^{n\times n}$ be a standard Gaussian random matrix. Then, for $0 < \delta < 1$,  
\[ \mathbb{P}\{ \|\mg\|_2 >  2\sqrt{n} + \sqrt{2\log(2/\delta)} \} \le \delta. \]
\end{theorem}
\begin{proof}
    From~\cite[Corollary 5.35]{Vershynin2009} follows for $t>0$
    \[\mathbb{P}\left\{ \|\mg\|_2  > 2\sqrt{n} + t\right\} \le 2\exp(-t^2/2).\]
    Set the right hand side equal to $\delta$ and solve for $t$. 
\end{proof}

\subsection{Proof of Theorem~\ref{thm:smoothed}}\label{app_a}
From $\rank(\mx) \ge k$ follows that $\rank(\mx_k^*)=k$.

\paragraph{Step 1: Concentration bounds} Let $\mathcal{E} = \mathcal{E}_1 \cap \mathcal{E}_2$ be an event, where 
\[ \mathcal{E}_1 \equiv \{ \sigma_n(\mz) \ge \Delta_1\}, \qquad  \mathcal{E}_2 \equiv \{\|\mg\|_2 \le \Delta_2\}. \] 
Theorems~\ref{thm:gaussthm1} and~\ref{thm:gaussthm2} imply 
that the complementary events occur with probability
$\mathbb{P}\{\mathcal{E}_1^c\} \le \delta/2$ and $\mathbb{P}\{\mathcal{E}_2^c\} \le \delta/2$. A union bound 
shows that $\mathbb{P}\{\mathcal{E}^c\}\equiv\mathbb{P}\{\mathcal{E}_1^c \cup \mathcal{E}_2^c\} \le \delta$. Thus, $\mathbb{P}(\mathcal{E}) \ge 1 - \delta$. 
The rest of the proof is conditioned on the event $\mathcal{E}$. 

\paragraph{Step 2: Greedy bound for $\mz_k$} Conditioned on the event $\mathcal{E}$, the matrix $\mz$ is nonsingular and $\rank(\mz_k)=k$. Corollary~\ref{c_2} implies 
\[ 0 \le \log\vol(\mz_k^*) - \log\vol(\mz_k) \le \frac1e \log\vol(\mz_k^*/\sigma_n(\mz)). \]
 
\paragraph{Step 3: Bound for $\mz_k^*$} Denote by $\mz_k' \equiv \mx_k^* + \epsilon\mg_k'$ the columns of $\mz$ corresponding to $\mx_k^*$, where $\mg_k'$ represents the columns of $\mg$ corresponding to $\mx_k^*$. The assumptions on $\epsilon$ imply for the event $\mathcal{E}$

\[ \tau \equiv \frac{\|\mg_k'\|_2}{\sigma_k(\mx_k^*)} \le \frac{\|\mg\|_2}{\sigma_k(\mx_k^*) } \le \frac{\epsilon\Delta_2}{\sigma_k(\mx_k^*)} = \eta_2 \le \frac12.  \] 
From the optimality of $\mz_k^*$, Lemma~\ref{lemma:logvolperturb},
and conditioning on~$\mathcal{E}$ follows
\[ \log\vol(\mz_k^*) \ge \log\vol(\mz_k') \ge \log\vol(\mx_k^*) -k\eta_2.\]
Similarly, we derive the upper  bound $\log\vol(\mz_k^*) \le \log(\mx_k^*) + k\eta_2.$

\paragraph{Step 4: Rank of~$\mx_k$} Since $\mz_k$ is the output of Algorithm~\ref{alg_qrcp} applied to $\mz$, 
we have \cite[Theorem 7.2]{GuEis96} 
\[\sigma_k(\mz_k) \ge \frac{\sigma_k(\mz)}{2^k\sqrt{n-k}} \ge \frac{\sigma_k(\mx) - \epsilon\Delta_2}{2^k\sqrt{n-k}}.\]
where second inequality follows from~\eqref{eqn:weyl},
\begin{equation*}
|\sigma_k(\mz)-\sigma_k(\mx)|\le \epsilon\|\mg\|_2 \le \epsilon\Delta_2,
\end{equation*}
and conditioning on the event $\mathcal{E}$. 

Abbreviate $\zeta \equiv 2^k\sqrt{n-k}$. Applying Weyl's inequality~\eqref{eqn:weyl} once again, along with the lower bound for $\sigma_k(\mz_k)$ gives
\[ \sigma_k(\mx_k) \ge \sigma_k(\mz_k) - \epsilon\Delta_2 \ge \zeta^{-1}\sigma_k(\mx) - (1+ \zeta^{-1})\epsilon\Delta_2.\]
Since $\epsilon< \frac{\sigma_k(\mx)}{1 + \zeta} \le \frac{\sigma_k(\mx)}{\Delta_2(1 + 2\zeta)}$, we get $\rank(\mx_k)=k$,
conditioned on~$\mathcal{E}$. From $\rank(\mx) \ge k$ follows $\sigma_k(\mx) > 0$.

Now, let $\mz_k = \mx_k + \epsilon\mg_k$. From  the assumptions on $\epsilon$ follows,
\[ \tau \equiv \frac{\|\mg_k\|_2}{\sigma_k(\mx_k)} \le \frac{\epsilon\Delta_2}{\sigma_k(\mz_k)} = \eta_1 \le   \frac{\epsilon\Delta_2}{(\sigma_k(\mx) - \epsilon\Delta_2)/\zeta}  \le  \frac12.\]
Solving for $\epsilon$ gives $\epsilon \le \frac{\sigma_k(\mx)}{\Delta_2(1+2\zeta)}$. Lemma~\ref{lemma:logvolperturb} 
implies that, conditioned on the event~$\mathcal{E}$,
\[  \log\vol(\mx_k) \ge \log\vol(\mz_k) -  k\eta_1 . \] 

\paragraph{Step 5: Putting everything together} From Step 3, optimality of $\mx_k^*$, and conditioning on the event $\mathcal{E}$ follows
\[ \begin{aligned} 0 \le   \log\vol(\mx_k^*)- \log\vol(\mx_k) \le  & \>\log\vol(\mz_k^*)  + k\eta_2 - \log\vol(\mx_k) \\ \le & \>  \log\vol(\mz_k^*)  - \log\vol(\mz_k) + k(\eta_2 + \eta_1).  \end{aligned}\]
Step 2 implies 
\[0 \le   \log\vol(\mx_k^*)- \log\vol(\mx_k) \le  \frac1e \log\vol(\mz_k^*/\sigma_n(\mz)) + k(\eta_1 + \eta_2). \] 
Plug in the bound for $\sigma_n(\mz)$ from Step 1 and the inequalities $\log\vol(\mz_k^*) \le \log\vol(\mx_k^*) + k\eta_2$ and $1/e \le 1$. Finally, use the fact that $\mathbb{P}(\mathcal{E}) \ge 1 -\delta$.

\subsection{Proof of Theorem~\ref{thm:smoothed2}}\label{app_c}
 The proof is the same as that of Theorem~\ref{thm:smoothed}, but uses Corollary~\ref{c_3} in Step 2, and~\cite[Theorem 3.2]{GuEis96} in Step 4.

\subsection{Proof of Theorem~\ref{thm:smoothed_psd}}\label{app_b}
Weyl's inequality (\ref{eqn:weyl2})
and the symmetric positive semidefiniteness of
$\ma$ imply
\[ \lambda_n(\mh) \ge \lambda_n(\ma) + \epsilon\lambda_n(\mg\mg^T) \ge \epsilon \sigma_n(\mg)^2. \]
Given the event $\mathcal{E}_1 \equiv \{\lambda_n(\mh) \ge \Delta_1 \}$, its complement occurs with probability
$\mathbb{P}(\mathcal{E}_1^c)\le \delta/2$
\cite[Theorem 3.3]{sankar2006smoothed}.
Similarly, for the event $\mathcal{E}_2 \equiv \{\| \mg\mg^T \|_2 \le \Delta_2 \}$, the complement occurs with probability $\mathbb{P}(\mathcal{E}_2^c)\le \delta/2$.
Thus, $\mathcal{E} = \mathcal{E}_1 \cap \mathcal{E}_2$ 
occurs with probability $\mathcal{P}(\mathcal{E}) \ge 1 -\delta$. 

The rest of the proof is similar to that of Theorem~\ref{thm:smoothed}. 

\subsection{Proof of Theorem~\ref{thm:smoothed_psd2}}\label{app_d}
 The proof is similar to that of Theorem~\ref{thm:smoothed_psd} but uses the analysis in~\cite[Section 4]{GuEis96}.  

%% file: ssbib.bib
@article {duersch2017randomized,
    AUTHOR = {Duersch, Jed A. and Gu, Ming},
     TITLE = {Randomized {QR} with column pivoting},
   JOURNAL = {SIAM J. Sci. Comput.},
  FJOURNAL = {SIAM Journal on Scientific Computing},
    VOLUME = {39},
      YEAR = {2017},
    NUMBER = {4},
     PAGES = {C263--C291},
      ISSN = {1064-8275,1095-7197},
          DOI = {10.1137/15M1044680},
}

@book {GovL13,
    AUTHOR = {Golub, G. H. and {Van Loan}, C. F.},
     TITLE = {Matrix computations},
    SERIES = {Johns Hopkins Studies in the Mathematical Sciences},
   EDITION = {Fourth},
 PUBLISHER = {Johns Hopkins University Press, Baltimore, MD},
      YEAR = {2013}
}

@article{taskar2013determinantal,
  title={Determinantal point processes for machine learning},
  author={Kulesza, B and Taskar, A.},
  journal={Found. Trends Mach. Learn.},
  volume={5 (2-3)},
  pages={123-286},
  year={2012},
}

@article{kelmans1983multiplicative,
  title={Multiplicative submodularity of a matrix's principal minor as a function of the set of its rows and some combinatorial applications},
  author={Kelmans, A. K. and Kimelfeld, B. N.},
  journal={Discrete Math.},
  volume={44},
  number={1},
  pages={113--116},
  year={1983},
  publisher={Elsevier}
}

@article{damle2025estimating,
  title={Estimating a matrix's singular values with interpolative decompositions},
  author={Damle, A. and Glas, S. and Townsend, A. and Yu, A.},
  journal={Linear Algebra Appl.},
  volume = {731},
  pages = {306-342},
  year={2026},
}

@article{chen2026optimal,
  title={Optimal Sensor Placement in {G}aussian Processes via Column Subset Selection},
  author={Chen, J. and Ji, H. and Saibaba, A. K.},
  journal={arXiv:2601.20781},
  year={2026}
}

@article{chen2025randomly,
  title={Randomly pivoted {C}holesky: Practical approximation of a kernel matrix with few entry evaluations},
  author={Chen, Y. and Epperly, E. N. and Tropp, J. A. and Webber, R. J.},
  journal={Comm. Pure Appl. Math.},
  volume={78},
  number={5},
  pages={995--1041},
  year={2025},
  publisher={Wiley Online Library}
}

@inproceedings{altschuler2016greedy,
  title={Greedy column subset selection: New bounds and distributed algorithms},
  author={Altschuler, Jason and Bhaskara, Aditya and Fu, Gang and Mirrokni, Vahab and Rostamizadeh, Afshin and Zadimoghaddam, Morteza},
  booktitle={International conference on machine learning},
  pages={2539--2548},
  year={2016},
  organization={PMLR}
}

@article{demmel2015communication,
  title={Communication avoiding rank revealing {QR} factorization with column pivoting},
  author={Demmel, James W and Grigori, Laura and Gu, Ming and Xiang, Hua},
  journal={SIAM J. Matrix Anal. Appl},
  volume={36},
  number={1},
  pages={55--89},
  year={2015},
  publisher={SIAM}
}

@article {gu2004strong,
    AUTHOR = {Gu, M. and Miranian, L.},
     TITLE = {Strong rank revealing {C}holesky factorization},
   JOURNAL = {Electron. Trans. Numer. Anal.},
  FJOURNAL = {Electronic Transactions on Numerical Analysis},
    VOLUME = {17},
      YEAR = {2004},
     PAGES = {76--92},
      ISSN = {1068-9613},
   MRCLASS = {65F05},
  MRNUMBER = {2040798},
}

@article{fornace2024column,
  title={Column and row subset selection using nuclear scores: algorithms and theory for {N}ystr\"om approximation, {CUR} decomposition, and graph {L}aplacian reduction},
  author={Fornace, Mark and Lindsey, Michael},
  journal={arXiv preprint arXiv:2407.01698},
  year={2024}
}

@inproceedings{mirzasoleiman2015lazier,
  title={Lazier than lazy greedy},
  author={Mirzasoleiman, B. and Badanidiyuru, A. and Karbasi, A. and Vondr{\'a}k, J. and Krause, A.},
  booktitle={Proceedings of the AAAI Conference on Artificial Intelligence},
  volume={29},
   year={2015}
}

@article{epperly2025embrace,
  title={Embrace rejection: Kernel matrix approximation by accelerated randomly pivoted {C}holesky},
  author={Epperly, Ethan N and Tropp, Joel A and Webber, Robert J},
  journal={SIAM J. Matrix Anal. Appl},
  volume={46},
  number={4},
  pages={2527--2557},
  year={2025},
  publisher={SIAM}
}

@book{horn2005basic,
TITLE = {The {S}chur complement and its applications},
SERIES = {Numerical Methods and Algorithms},
    VOLUME = {4},
    EDITOR = {Zhang, Fuzhen},
 PUBLISHER = {Springer-Verlag, New York},
      YEAR = {2005},
     PAGES = {xvi+295},
      ISBN = {0-387-24271-6},
        DOI = {10.1007/b105056},
       URL = {https://doi.org/10.1007/b105056},
}

@article{Vershynin2009,
author = {Vershynin, R.},
doi = {10.1017/CBO9780511794308.006},
eprint = {1011.3027},
isbn = {9780511794308},
journal = {Compressed Sensing: Theory and Applications},
pages = {210--268},
title = {{Introduction to the non-asymptotic analysis of random matrices}},
year = {2009}
}

@book {HoJoI,
    AUTHOR = {Horn, R. A. and Johnson, C. R.},
     TITLE = {Matrix analysis},
   EDITION = {Second},
 PUBLISHER = {Cambridge University Press, Cambridge},
      YEAR = {2013}}

@book {Higham2002,
    AUTHOR = {Higham, Nicholas J.},
     TITLE = {Accuracy and stability of numerical algorithms},
   EDITION = {Second},
 PUBLISHER = {Society for Industrial and Applied Mathematics (SIAM),
              Philadelphia, PA},
      YEAR = {2002},
     PAGES = {xxx+680},
      ISBN = {0-89871-521-0},
       DOI = {10.1137/1.9780898718027},
       URL = {https://doi.org/10.1137/1.9780898718027}
}

@article {CM09,
    AUTHOR = {{\c C}ivril, A. and Magdon-Ismail, M.},
     TITLE = {On selecting a maximum volume sub-matrix of a matrix and
              related problems},
   JOURNAL = {Theoret. Comput. Sci.},
    VOLUME = {410},
      YEAR = {2009},
    NUMBER = {47-49},
     PAGES = {4801--4811},
      ISSN = {0304-3975,1879-2294},
        DOI = {10.1016/j.tcs.2009.06.018},
       URL = {https://doi.org/10.1016/j.tcs.2009.06.018},
}

@article {GuEis96,
    AUTHOR = {Gu, M. and Eisenstat, S. C.},
     TITLE = {Efficient algorithms for computing a strong rank-revealing
              {QR} factorization},
   JOURNAL = {SIAM J. Sci. Comput.},
    VOLUME = {17},
      YEAR = {1996},
    NUMBER = {4},
     PAGES = {848--869},
      ISSN = {1064-8275},
       DOI = {10.1137/0917055},
       URL = {https://doi.org/10.1137/0917055},
}

@article {SX2016,
    AUTHOR = {Shin, Y. and Xiu, D.},
     TITLE = {Nonadaptive quasi-optimal points selection for least squares
              linear regression},
   JOURNAL = {SIAM J. Sci. Comput.},
    VOLUME = {38},
      YEAR = {2016},
    NUMBER = {1},
     PAGES = {A385--A411},
      ISSN = {1064-8275,1095-7197},
       DOI = {10.1137/15M1015868},
       URL = {https://doi.org/10.1137/15M1015868},
}

@article {LCS2024,
AUTHOR = {Lauzon, J. T. and Cheung, S. W. and Shin, Y. and Choi, Y. and Copeland, D. M. and Huynh, K.},
     TITLE = {S-{OPT}: a points selection algorithm for hyper-reduction in
              reduced order models},
   JOURNAL = {SIAM J. Sci. Comput.},
     VOLUME = {46},
      YEAR = {2024},
    NUMBER = {4},
     PAGES = {B474--B501},
      ISSN = {1064-8275,1095-7197},
         DOI = {10.1137/22M1484018},
       URL = {https://doi.org/10.1137/22M1484018},
}

@article {Nemhauser1978,
    AUTHOR = {Nemhauser, G. L. and Wolsey, L. A. and Fisher, M. L.},
     TITLE = {An analysis of approximations for maximizing submodular set
              functions. {I}},
   JOURNAL = {Math. Programming},
     VOLUME = {14},
      YEAR = {1978},
    NUMBER = {3},
     PAGES = {265--294},
      ISSN = {0025-5610,1436-4646},
        DOI = {10.1007/BF01588971},
       URL = {https://doi.org/10.1007/BF01588971},
}

@article {BusGo65,
    AUTHOR = {Businger, Peter and Golub, Gene H.},
     TITLE = {{L}inear least squares
              solutions by {H}ouseholder transformations},
   JOURNAL = {Numer. Math.},
     VOLUME = {7},
      YEAR = {1965},
     PAGES = {269--276},
      ISSN = {0029-599X,0945-3245},
         DOI = {10.1007/BF01436084},
       URL = {https://doi.org/10.1007/BF01436084},
}

@article{DasK2018,
  author  = {Das, A. and Kempe, D.},
  title   = {Approximate Submodularity and its Applications: Subset Selection, Sparse Approximation and Dictionary Selection},
  journal = {J. Mach. Learn. Res.},
  year    = {2018},
  volume  = {19},
  number  = {3},
  pages   = {1--34},
  url     = {http://jmlr.org/papers/v19/16-534.html}
}

@article{Corne1977,
author = {Cornuejols, Gerard and Fisher, Marshall L. and Nemhauser, George L.},
title = {Location of Bank Accounts to Optimize Float: An Analytic Study of Exact and Approximate Algorithms},
journal = {Management Sci.},
volume = {23},
number = {8},
pages = {789-810},
year = {1977},
doi = {10.1287/mnsc.23.8.789},
URL = {https://doi.org/10.1287/mnsc.23.8.789},
eprint = {https://doi.org/10.1287/mnsc.23.8.789}
}

@article{sankar2006smoothed,
  title={Smoothed analysis of the condition numbers and growth factors of matrices},
  author={Sankar, Arvind and Spielman, Daniel A and Teng, Shang-Hua},
  journal={SIAM Journal on Matrix Analysis and Applications},
  volume={28},
  number={2},
  pages={446--476},
  year={2006},
  publisher={SIAM}
}

@inproceedings {DEFM15,
    AUTHOR = {Di Summa, M. and Eisenbrand, F. and Faenza, Y. and
              Moldenhauer, C.},
     TITLE = {On largest volume simplices and sub-determinants},
 BOOKTITLE = {Proceedings of the {T}wenty-{S}ixth {A}nnual {ACM}-{SIAM}
              {S}ymposium on {D}iscrete {A}lgorithms},
     PAGES = {315--323},
 PUBLISHER = {SIAM, Philadelphia, PA},
      YEAR = {2015},
       DOI = {10.1137/1.9781611973730.23},
       URL = {https://doi.org/10.1137/1.9781611973730.23},
}

@article {Massei2022,
    AUTHOR = {Massei, S.},
TITLE = {Some algorithms for maximum volume and cross approximation of symmetric semidefinite matrices},
   JOURNAL = {BIT},
      VOLUME = {62},
      YEAR = {2022},
    NUMBER = {1},
     PAGES = {195--220},
      ISSN = {0006-3835,1572-9125},
        DOI = {10.1007/s10543-021-00872-1},
       URL = {https://doi.org/10.1007/s10543-021-00872-1},
}

@article{osinsky2026subsetselectionmatricescolumn,
      title={Subset selection for matrices by column exchange}, 
      author={Osinsky, Alexander  and Kozyrev, Ivan },
      year={2026},
      journal={arXiv preprint arXiv:2604.14418}, 
}

@article {nemhauser1978best,
    AUTHOR = {Nemhauser, G. L. and Wolsey, L. A.},
     TITLE = {Best algorithms for approximating the maximum of a submodular
              set function},
   JOURNAL = {Math. Oper. Res.},
    VOLUME = {3},
      YEAR = {1978},
    NUMBER = {3},
     PAGES = {177--188},
      ISSN = {0364-765X,1526-5471},
       DOI = {10.1287/moor.3.3.177},
}

@incollection{goreinov2010find,
  title={How to find a good submatrix},
  author={Goreinov, Sergei A and Oseledets, Ivan V and Savostyanov, Dimitry V and Tyrtyshnikov, Eugene E and Zamarashkin, Nikolay L},
  booktitle={Matrix Methods: Theory, Algorithms And Applications: Dedicated to the Memory of Gene Golub},
  pages={247--256},
  year={2010},
  publisher={World Scientific}
}

@article{grigori2025randomized,
  title={Randomized strong rank-revealing {QR} for column subset selection and low-rank matrix approximation},
  author={Grigori, Laura and Xue, Zhipeng},
  journal={arXiv preprint arXiv:2503.18496},
  year={2025}
}

@inproceedings{Mirza2013,
 author = {Mirzasoleiman, B. and Karbasi, A. and Sarkar, R. and Krause, A.},
 booktitle = {Advances in Neural Information Processing Systems},
 editor = {C.J. Burges and L. Bottou and M. Welling and Z. Ghahramani and K. Weinberger},
 pages = {},
 publisher = {Curran Associates, Inc.},
 title = {Distributed Submodular Maximization: Identifying Representative Elements in Massive Data},
 url = {https://proceedings.neurips.cc/paper_files/paper/2013/file/84d2004bf28a2095230e8e14993d398d-Paper.pdf},
 volume = {26},
 year = {2013}
}

@inproceedings{ChenZhang2018,
author = {Chen, L. and Zhang, G. and Zhou, H.},
title = {Fast greedy {MAP} inference for determinantal point process to improve recommendation diversity},
year = {2018},
publisher = {Curran Associates Inc.},
address = {Red Hook, NY, USA},
booktitle = {Proceedings of the 32nd International Conference on Neural Information Processing Systems},
pages = {5627–5638},
numpages = {12},
location = {Montr\'{e}al, Canada},
series = {NIPS'18}
}

@incollection {Krause2014,
    AUTHOR = {Krause, A. and Golovin, D.},
     TITLE = {Submodular function maximization},
 BOOKTITLE = {Tractability},
     PAGES = {71--104},
 PUBLISHER = {Cambridge Univ. Press, Cambridge},
      YEAR = {2014},
      ISBN = {978-1-107-02519-6},
}

@article {Harb2012,
    AUTHOR = {Harbrecht, H. and Peters, M. and Schneider, R.},
     TITLE = {On the low-rank approximation by the pivoted {C}holesky
              decomposition},
   JOURNAL = {Appl. Numer. Math.},
    VOLUME = {62},
      YEAR = {2012},
    NUMBER = {4},
     PAGES = {428--440},
      ISSN = {0168-9274,1873-5460},
         DOI = {10.1016/j.apnum.2011.10.001},
       URL = {https://doi.org/10.1016/j.apnum.2011.10.001},
}

@article {Beb2000,
    AUTHOR = {Bebendorf, M.},
     TITLE = {Approximation of boundary element matrices},
   JOURNAL = {Numer. Math.},
    VOLUME = {86},
      YEAR = {2000},
    NUMBER = {4},
     PAGES = {565--589},
      ISSN = {0029-599X,0945-3245},
       DOI = {10.1007/PL00005410},
       URL = {https://doi.org/10.1007/PL00005410},
}

@article {Osinsky23,
    AUTHOR = {Osinsky, A.},
     TITLE = {Volume-based subset selection},
   JOURNAL = {Numer. Linear Algebra Appl.},
     VOLUME = {31},
      YEAR = {2024},
    NUMBER = {1},
     PAGES = {Paper No. e2525, 14},
      ISSN = {1070-5325,1099-1506},
   MRCLASS = {65F20 (65F55)},
  MRNUMBER = {4677350},
       DOI = {10.1002/nla.2525},
       URL = {https://doi.org/10.1002/nla.2525},
}

@article {CKM2020,
    AUTHOR = {Cortinovis, A. and Kressner, D. and Massei, S.},
     TITLE = {On maximum volume submatrices and cross approximation for
              symmetric semidefinite and diagonally dominant matrices},
   JOURNAL = {Linear Algebra Appl.},
      VOLUME = {593},
      YEAR = {2020},
     PAGES = {251--268},
      ISSN = {0024-3795,1873-1856},
       DOI = {10.1016/j.laa.2020.02.010},
       URL = {https://doi.org/10.1016/j.laa.2020.02.010},
}

@article {Papa84,
    AUTHOR = {Papadimitriou, Christos H.},
     TITLE = {The largest subdeterminant of a matrix},
   JOURNAL = {Bull. Soc. Math. Gr\`ece (N.S.)},
      VOLUME = {25},
      YEAR = {1984},
     PAGES = {95--105},
      ISSN = {0373-1391},
  }

@article{Ben1992,
    AUTHOR = {Ben-Israel, A.},
     TITLE = {A volume associated with {$m\times n$} matrices},
      NOTE = {Sixth Haifa Conference on Matrix Theory (Haifa, 1990)},
   JOURNAL = {Linear Algebra Appl.},
     VOLUME = {167},
      YEAR = {1992},
     PAGES = {87--111},
      ISSN = {0024-3795,1873-1856},
         DOI = {10.1016/0024-3795(92)90340-G},
       URL = {https://doi.org/10.1016/0024-3795(92)90340-G},
}

@incollection {Bodek22,
    AUTHOR = {Bodek, K. and Feldman, M.},
     TITLE = {Maximizing sums of non-monotone submodular and linear
              functions: understanding the unconstrained case},
 BOOKTITLE = {30th annual {E}uropean {S}ymposium on {A}lgorithms},
    SERIES = {LIPIcs. Leibniz Int. Proc. Inform.},
    VOLUME = {244},
     PAGES = {Art. No. 23, 17},
 PUBLISHER = {Schloss Dagstuhl. Leibniz-Zent. Inform., Wadern},
      YEAR = {2022},
      ISBN = {978-3-95977-247-1},
          DOI = {10.4230/lipics.esa.2022.23},
       URL = {https://doi.org/10.4230/lipics.esa.2022.23},
}

@inproceedings {Buch14,
    AUTHOR = {Buchbinder, N. and Feldman, M. and Naor, J. and
              Schwartz, R.},
     TITLE = {Submodular maximization with cardinality constraints},
 BOOKTITLE = {Proceedings of the {T}wenty-{F}ifth {A}nnual {ACM}-{SIAM}
              {S}ymposium on {D}iscrete {A}lgorithms},
     PAGES = {1433--1452},
 PUBLISHER = {ACM, New York},
      YEAR = {2014},
      ISBN = {978-1-611973-38-9},
         DOI = {10.1137/1.9781611973402.106},
       URL = {https://doi.org/10.1137/1.9781611973402.106},
}

@inproceedings {Buch25,
    AUTHOR = {Buchbinder, N. and Feldman, M.},
     TITLE = {Extending the extension: deterministic algorithm for
              non-monotone submodular maximization},
 BOOKTITLE = {S{TOC}'25---{P}roceedings of the 57th {A}nnual {ACM}
              {S}ymposium on {T}heory of {C}omputing},
     PAGES = {1130--1141},
 PUBLISHER = {ACM, New York},
      YEAR = {2025},
      ISBN = {979-8-4007-1510-5},
   MRCLASS = {68R05},
  MRNUMBER = {4928504},
       DOI = {10.1145/3717823.3718120},
       URL = {https://doi.org/10.1145/3717823.3718120},
}

@inproceedings{Hemmi22,
author = {Hemmi, S. and Oki, T. and Sakaue, S. and Fujii, K. and Iwata, S.},
title = {Lazy and fast greedy {MAP} inference for determinantal point process},
year = {2022},
isbn = {9781713871088},
publisher = {Curran Associates Inc.},
address = {Red Hook, NY, USA},
booktitle = {Proceedings of the 36th International Conference on Neural Information Processing Systems},
articleno = {201},
numpages = {14},
location = {New Orleans, LA, USA},
series = {NIPS '22}
}

@inproceedings{Gillen12,
 author = {Gillenwater, Jennifer and Kulesza, Alex and Taskar, Ben},
 booktitle = {Advances in Neural Information Processing Systems},
 editor = {F. Pereira and C.J. Burges and L. Bottou and K. Weinberger},
 pages = {},
 publisher = {Curran Associates, Inc.},
 title = {Near-Optimal {MAP} Inference for Determinantal Point Processes},
 url = {https://proceedings.neurips.cc/paper_files/paper/2012/file/6c8dba7d0df1c4a79dd07646be9a26c8-Paper.pdf},
 volume = {25},
 year = {2012}
}

@article{GoK11,
author = {Golovin, Daniel and Krause, Andreas},
title = {Adaptive submodularity: theory and applications in active learning and stochastic optimization},
year = {2011},
issue_date = {September 2011},
publisher = {AI Access Foundation},
address = {El Segundo, CA, USA},
volume = {42},
number = {1},
issn = {1076-9757},
journal = {J. Artif. Int. Res.},
month = sep,
pages = {427–486},
numpages = {60}
}

@inproceedings{GoK10,
	author = {Daniel Golovin and Andreas Krause},
	booktitle = {Proc. Conference on Learning Theory (COLT)},
	title = {Adaptive Submodularity: A New Approach to Active Learning and Stochastic Optimization},
	year = {2010}}

@article{MikOs18,
title = {Rectangular maximum-volume submatrices and their applications},
journal = {Linear Algebra Appl.},
volume = {538},
pages = {187-211},
year = {2018},
issn = {0024-3795},
doi = {https://doi.org/10.1016/j.laa.2017.10.014},
author = {A. Mikhalev and I.V. Oseledets},
}

@article{Rob89,
author = {Robertazzi, T. G. and Schwartz, S. C.},
title = {An Accelerated Sequential Algorithm for Producing {D}-Optimal Designs},
journal = {SIAM J. Sci. Stat. Comput.},
volume = {10},
number = {2},
pages = {341-358},
year = {1989},
doi = {10.1137/0910022},
}

@InProceedings{Sakaue20,
title =  {Guarantees of Stochastic Greedy Algorithms for Non-monotone Submodular Maximization with Cardinality Constraints},  
author = {Sakaue, S.},  
booktitle =  {Proceedings of the Twenty Third International Conference on Artificial Intelligence and Statistics},  
pages =  {11--21},  
year =  {2020},  
volume =  {108},  
month =  {26--28 Aug},  
}

@article{Hash2020,
author = {Hashemi, A. and Ghasemi, M. and Vikalo, H. and Topcu, U.},
year = {2020},
month = {03},
pages = {1-1},
title = {Randomized Greedy Sensor Selection: Leveraging Weak Submodularity},
volume = {PP},
journal = {IEEE Transactions on Automatic Control},
doi = {10.1109/TAC.2020.2980924}
}

@inproceedings{Sham2010,
  title = {Greedy sensor selection: Leveraging submodularity},
  author = {M. Shamaiah and S. Banerjee and H. Vikalo},
  year = {2010},
  doi = {10.1109/CDC.2010.5717225},
  url = {http://dx.doi.org/10.1109/CDC.2010.5717225},
   pages = {2572-2577},
  booktitle = {Proceedings of the 49th IEEE Conference on Decision and Control, CDC 2010, December 15-17, 2010, Atlanta, Georgia, USA},
  publisher = {IEEE},
}

@article{Cali11,
author = {Calinescu, G. and Chekuri, C. and P\'{a}l, M. and Vondr\'{a}k, J.},
title = {Maximizing a Monotone Submodular Function Subject to a Matroid Constraint},
journal = {SIAM J. Comput.},
volume = {40},
number = {6},
pages = {1740-1766},
year = {2011},
doi = {10.1137/080733991},
}

@article{Zhang2016,
  author={Zhang, Zhenliang and Chong, Edwin K. P. and Pezeshki, Ali and Moran, William},
  journal={IEEE Transactions on Automatic Control}, 
  title={String Submodular Functions With Curvature Constraints}, 
  year={2016},
  volume={61},
  number={3},
  pages={601-616},
  doi={10.1109/TAC.2015.2440566}}

@techreport{Street2007,
author = "Streeter, M. and Golovin, D.",
title = "An online algorithm for maximizing submodular functions",
number = "CMU-CS-07-171", 
institution = "Carnegie Mellon University", 
year = "2007"
}
